\documentclass[11pt]{amsart}

\usepackage[a4paper,margin=30mm]{geometry}

\usepackage[foot]{amsaddr}
\usepackage{graphicx}
\usepackage{amsmath,amsthm,amssymb,amsfonts}
\usepackage{dsfont} 
\usepackage{bm}
\usepackage[square,numbers]{natbib}

\usepackage{mleftright}\mleftright 

\numberwithin{equation}{section}
\numberwithin{figure}{section}

\newcommand{\1}{\mathds{1}}
\newcommand\dd{\mathrm{d}}

\newcommand{\C}{\mathbb{C}}
\newcommand{\E}{\mathbb{E}}
\renewcommand{\Pr}{\mathbb{P}}
\newcommand{\R}{\mathbb{R}}
\newcommand{\Cov}{\mathrm{Cov}}
\newcommand{\diag}{\mathrm{diag}}
\newcommand{\erf}{\mathtt{erf}}
\newcommand{\erfc}{\mathtt{erfc}}
\newcommand{\ind}{{\mathrm{ind}}}
\newcommand{\pf}{\mathrm{pf}}
\newcommand{\sgn}{\mathrm{sgn}}
\newcommand{\tr}{\mathrm{tr}}
\newcommand{\Var}{\mathrm{Var}}

\newtheorem{theorem}{Theorem}[section]

\newtheorem{lemma}[theorem]{Lemma}
\newtheorem{proposition}[theorem]{Proposition}
\newtheorem{remark}[theorem]{Remark}

\begin{document}

\title[]{
Expected number density of critical points of smooth Gaussian random fields in arbitrary dimensions}

\author{Satoshi Kuriki\textsuperscript{1,2}}
\address{\textsuperscript{1}The Institute of Statistical Mathematics}
\address{\textsuperscript{2}The Graduate University for Advanced Studies (SOKENDAI)}
\email{kuriki@ism.ac.jp}

\author{Takahiko Matsubara\textsuperscript{3,2}}
\address{\textsuperscript{3}Institute of Particle and Nuclear Studies, High Energy Accelerator Research Organization (KEK)}
\email{tmats@post.kek.jp}

\author{Satoshi Iso\textsuperscript{4,3,2}}
\address{\textsuperscript{4}RIKEN Center for Interdisciplinary Theoretical and Mathematical Sciences (iTHems)}
\email{iso@post.kek.jp}


\maketitle

\begin{abstract}
We obtain explicit formulas for the expected number and height distribution of critical points of smooth isotropic Gaussian random fields on $\R^d$.
The expected number density formula is expressed in terms of at most one-dimensional integrals, regardless of the dimension $d$.
To obtain the formulas, we provide a variant of de Bruijn's theorem,
as well as Weierstrass' convolution formula with a Gaussian random variable and its inversion.
\end{abstract}

\section{Introduction}

Let $X(x)$, $x=(x^i)_{1\le i\le d}\in\R^d$ be an isotropic smooth Gaussian random field on the $d$-dimensional Euclidean space.
We assume that the map $x\mapsto X(x)$ is of class $C^2$.
The Gaussian random field is characterized by its mean and covariance.
Because of the isotropic assumption, we assume without loss of generality that the mean function is constant and the covariance between two points $x$ and $y$ is a function of $\Vert x-y\Vert$:
\begin{equation}
\label{e-cov}
 \E[X(x)]=0, \quad \E[X(x) X(y)]=C\bigl(\tfrac{1}{2}\Vert x-y\Vert^2\bigr), \quad C(0)=1.
\end{equation}
We assume regularity conditions on the smoothness of $C(\cdot)$ so that the random field $x\mapsto X(x)$ is smooth enough and has discrete critical points almost surely.

The set of critical points of index $m$ is
\[
 CP(m) = \bigl\{ x \in \R^d \mid \nabla X(x)=0,\,\ind\bigl(\nabla^2 X(x)\bigr)=m \bigr\},
\] 
where $\nabla X=(X_i)_{1\le i\le d}$, $X_i=\partial X/\partial x^i$,
$\nabla^2 X(x)=(X_{ij}(x))_{1\le i,j\le d}$, $X_{ij}=\partial^2 X/\partial x^i\partial x^j$,
and $\ind(\cdot)$ is the number of negative eigenvalues.
Let $V$ be an arbitrary domain in $\R^d$ having the volume $|V|=\int_V\dd x$.
We aim to compute the \textit{expected number density} $f_{m,d}(\nu)$ of the critical points of index $m$ defined by
\begin{equation}
\label{fm}
 \int_\nu^\infty f_{m,d}(v)\,\dd v = \frac{1}{|V|}\,
 \E \Bigl[ \#\bigl\{x\in CP(m)\cap V, \, X(x)>\nu \bigr\} \Bigr],
\end{equation}
where
$\#\{\cdot\}$ is the cardinality of a finite set.
Because of the isotropic property, the right-hand side is proportional to the volume of $V$.
The integrand of the left-hand side $f_{m,d}(v)\dd v$ is interpreted as the expected number density of critical points of index $m$ in the level set
\begin{equation}
\label{levelset}
 \bigl\{ x\in \R^d \,|\, X(x)\in (\nu,\nu+\dd\nu) \bigr\}.
\end{equation}

The normalized $f_{m,d}(\nu)$ is interpreted as the height density function of $X(x^*)$ at the critical points $x^*\in CP(m)$:
\begin{equation*}
 \frac{\int_{\nu}^\infty f_{m,d}(v)\,\dd v}{\int_{-\infty}^\infty f_{m,d}(v)\,\dd v}
 = \Pr \bigl(X(x^*)>\nu \,|\, x^*\in CP(m)\bigr).
\end{equation*}

\citet{Cheng-Schwartzman:2018} obtained an expression for $f_{m,d}(\nu)$ as an expectation with respect to the probability measure of the eigenvalues of $d\times d$ Gaussian orthogonally invariant (GOI) random matrices.
The distribution of the GOI matrix depends on a parameter $\sigma$, whose feasible parameter space is (\ref{sigma}).
The Gaussian orthogonal ensemble (GOE) corresponds to the special case $\sigma=0$.
Although the case where $\sigma\ge 0$ looks easier to handle than the case where $\sigma<0$, \cite{Cheng-Schwartzman:2018} showed that the function $f_{m,d}(\nu)$ has no singularities at $\sigma=0$ by expressing $f_{m,d}(\nu)$ with the probability measure of the GOI.

However, their resulting formula needs $(d+2)$-fold integration.
They demonstrate that $f_{m,d}(\nu)$ for $d=2$ can be expressed by interms of special functions (complementary error function $\erfc$) without integrals.
However, when $d\ge 3$, it is not obvious how to simplify the multifold integration.

In this paper, we provide a formula for $f_{m,d}(\nu)$ that requires at most one-dimensional integral.
The integrand of the one-dimensional integral is explicitly expressed in terms of elementary functions and the error functions (when the dimension $d$ is arbitrary).
For this purpose, we evaluate the generating function $F_d(\nu;z)$ for $f_{m,d}(\nu)$ by
\[
 F_d(\nu;z) = \sum_{m=0}^d f_{m,d}(\nu)z^m.
\]
We will also see that $F_d(\nu;z)$ has simple forms when $z=\pm 1$.
$F_d(\nu;-1)\dd\nu$ is the alternating sum of the expected number densities of critical points in the set (\ref{levelset}).
This is referred to as the expected Euler number density and is explicitly known as
\[
 F_d(\nu;-1) = (-1)^d\Bigl(\frac{\gamma}{2\pi}\Bigr)^{d/2} H_d(\nu) \phi(\nu),
\]
where $\gamma=\Var(\partial X(x)/\partial x^1)$, and
\begin{equation}
\label{phi-H}
 \phi(\nu) = \frac{1}{\sqrt{2\pi}} e^{-\nu^2/2}
 \quad\mbox{and}\quad
 H_d(\nu) = \phi(\nu)^{-1}\left(-\frac{\dd}{\dd\nu}\right)^d \phi(\nu)
\end{equation}
are the probability density function of the standard Gaussian distribution $\mathcal{N}(0,1)$ and the Hermite polynomial of degree $d$, respectively
 \citep{Adler:1981,Tomita:1986}.

The study of critical points of random fields has been motivated by research in cosmology \cite{Adler:1981}.
In the standard picture of structure formation, high peaks of the primordial cosmological density field indicate the formation sites of cosmological structures,  such as galaxies, galaxy clusters, and other large-scale structures that subsequently emerge.
Therefore, the abundance and spatial distribution of local maxima of random fields have been used to analyze the statistical properties of cosmological structures \cite{BBKS:1986}.

In addition, peak and critical point statistics of random fields have various applications in cosmology.
They provide powerful tools for testing Gaussianity \cite{Gay-etal:2012,Matsubara:2020},
characterizing hot and cold spots in cosmic microwave background (CMB) observations \cite{Larson-Wandelt:2004},
and extracting cosmological information from weak-lensing maps \cite{Maturi-etal:2011}.
The combination of critical point statistics and Betti numbers has been proposed to reveal non-Gaussianity and regional topological differences \cite{Henderson-etal:2020}.

A successful application outside cosmology is the study of phase transitions.
The quantity $F_d(\nu;1)\dd\nu$ is the expected number of all critical points in the set (\ref{levelset}), and \citep{Fyodorov:2004} pointed out that the behavior of the number of critical points as the dimension $d$ tends to infinity reflects the spin-glass phase transition.
For subsequent studies of critical points and spin-glass phase transitions, see, e.g., \cite{Auffinger-etal:2013,Yamada-Vilenkin:2018}.
For other applications of critical points in various areas, such as oceanography and neuroimaging, see the comprehensive reference in \cite{Cheng-Schwartzman:2018}.

In this paper, we derive a formula for $F_{d}(\nu;z)$ that involves at most a one-dimensional integral.
Our strategy consists of two steps.
First, $F_{d}(\nu;z)$ is obtained in the cases where the distribution of the second derivative random field (i.e., Hessian process) coincides with that of a GOE matrix.
Second, we show that any $F_{d}(\nu;z)$ can be obtained as a Weierstrass transform or its inverse transform of the $F_{d}(\nu;z)$ for the GOE case.
The Weierstrass transform and its inverse are given by one-dimensional integrals.

The outline of this paper is as follows.
In Section \ref{sec:preliminary}, we summarize preliminary facts and present an expression for $F_{d}(\nu;z)$, which will be simplified in the next section.
In Section \ref{sec:main}, a simplified formula for $F_{d}(\nu;z)$ is obtained according to the two steps mentioned above.
The convolution formula with a Gaussian random variable referred to as the Weierstrass transform, and its inversion formula are provided in a form suitable for our application. 
In the Appendix, we prepare a variant of de Bruijn's theorem \cite{deBruijn:1955}.
A simple proof of \citet{Fyodorov:2004}'s formula is given by establishing identity (\ref{LdHKd}).
It is also shown that the $F_{d}(\nu;z)$ for $d=1,2$ can be expressed without using the integrals.

\section{Preliminary}
\label{sec:preliminary}

\subsection{The Kac-Rice formula with regularity conditions}
\label{subsec:kac-rice}

The standard technique for handling critical points is the Kac-Rice formula, or equivalently, the coarea formula.
We first explain it in an intuitive way, and then restate it as a mathematical proposition.

Let $g:\R^d\to\R^d$ be a smooth function whose zeros are isolated,
and let $h:\R^d\to\R$ be a continuous function.
Let $\delta^d(\cdot)$ be the delta function in $\R^d$.
The coarea formula states
\[
 \sum_{x\in V: g(x)=0} h(x) = \int_V h(x) \delta^d(g(x)) \left|\frac{\partial g(x)}{\partial x}\right| \,\dd x,
\]
where $\dd x$ is the Lebesgue measure on $\R^d$.
By letting $g(x)=\nabla X(x)$, the number of critical points in (\ref{fm}) is expressed as
\begin{align*}
\# \bigl\{ & x\in CP(m)\cap V,\, X(x)>\nu \bigr\} \\
&= \sum_{x\in V: \nabla X(x)=0} \1\{\ind(\nabla^2 X(x))=m,\ X(x)>\nu\} \\
&= \int_{V} \1\{\ind(\nabla^2 X(x))=m\} \1\{X(x)>\nu\} \delta^d(\nabla X(x)) |\det(\nabla^2 X(x))| \,\dd x \\
&= (-1)^m \int_{V} \1\{\ind(\nabla^2 X(x))=m\} \1\{X(x)>\nu\}
 \delta^d(\nabla X(x)) \det(\nabla^2 X(x)) \,\dd x.
\end{align*}
Noting that $\int_V \dd x=|V|$, and as will be seen below, that $(X(x),\nabla^2 X(x))$ and $\nabla X(x)$ are independent for each fixed $x$,
if the expectation $\E[\ ]$ and the integral $\int_V\dd x$ can be interchanged,
we have
\begin{equation}
\label{fmd} 
\begin{aligned}
 f_{m,d}(\nu)
&= -(-1)^m \frac{\dd}{\dd\nu}\E[ \det(\nabla^2 X) \1\{\ind(\nabla^2 X)=m\} \1\{X>\nu\} \,|\, \nabla X=0] p_{\nabla X}(0) \\
&= \frac{(-1)^m}{(2\pi\gamma)^{d/2}} \E[\det(\nabla^2 X) \1\{\ind(\nabla^2 X)=m\} \,|\, X=\nu] \phi(\nu),
\end{aligned}
\end{equation}
where
\[
 p_{\nabla X}(0) = \frac{1}{(2\pi\gamma)^{d/2}}
\]
is the probability density function of $\nabla X(x)$ evaluated at $\nabla X(x)=0$.
In (\ref{fmd}), the argument $x$ of $X$, $\nabla X$ and $\nabla^2 X$ is omitted since their distributions are independent of the point $x$.

The informal derivation above is justified under the regularity conditions,
which is fully stated in \citet{Adler-Taylor:2007}.
Throughout the paper, we impose the assumptions of Proposition \ref{prop:kac-rice} below.

\begin{proposition}
\label{prop:kac-rice}

Under the regularity conditions in Theorem 11.2.1 of \cite{Adler-Taylor:2007} with $f(\cdot)$ and $g(\cdot)$ replaced by $\nabla X(\cdot)$ and $(X(\cdot),\nabla^2 X(\cdot))$, respectively, formula (\ref{fmd}) holds.
\end{proposition}

\subsection{Hessian processes as random matrices}
\label{subsec:hessian}

Under the regularity conditions assumed in Section \ref{subsec:kac-rice},
by taking the derivatives of (\ref{e-cov}) with respect to $\partial/\partial x^i$,  $\partial/\partial y^j$, etc., and evaluating them at $x=y$, we obtain a series of identities.
For example, 
\[
 \E[X_i(x)X_j(x)] = \frac{\partial^2 \E[X(x)X(y)]}{\partial x^i\partial y^j}\bigg|_{x=y} = \frac{\partial^2 C\bigl(\frac{1}{2}\Vert x-y\Vert^2\bigr)}{\partial x^i\partial y^j} \bigg|_{x=y} = -\delta_{ij} C'(0).
\]
(See also the discussion in \citep[Section 2.2]{Kuriki-Matsubara:2023}.)
In summary, for each $x$, $(X(x),\nabla X,\nabla^2 X(x))$ is centered Gaussian with covariance
\[
\Cov\Bigl((X,X_i,X_{ij}),(X,X_k,X_{kl})\Bigr)
= \begin{pmatrix}
 1 & 0 & C'(0)\delta_{kl} \\
 0 & -C'(0)\delta_{ik} & 0 \\
 C'(0)\delta_{ij} & 0 & C''(0)(\delta_{ik}\delta_{jl}+\delta_{ij}\delta_{kl}+\delta_{il}\delta_{jk})
\end{pmatrix}.
\]
Therefore, for each $x$, $(X(x),\nabla^2 X(x))$ and $\nabla X(x)$ are independent.
It is shown that this matrix is positive semi-definite if and only if
\[
 C'(0) \le 0, \quad C''(0)\ge \frac{d}{d+2} C'(0)^2.
\]
\citep[Eqs.\,(2.7), (2.8)]{Kuriki-Matsubara:2023}.

Here, the case $C'(0)=0$ holds if and only if $X_i(x)\equiv 0$ a.s., meaning that $X(x)$ is a random variable independent of $x$.
We exclude this trivial case and assume $C'(0)<0$.
The other boundary case,
\[
 C'(0) < 0, \quad C''(0) = \frac{d}{d+2} C'(0)^2
\]
was shown to be attained by \citet[Example 3.12]{Cheng-Schwartzman:2018}.
This is a degenerate case where $\sum_{i=1}^d X_{ii}(x)$ is proportional to $X(x)$.

In what follows, we fix a point $x$ and omit the argument $x$ from $X(x)$, $\nabla X(x)$, and $\nabla^2 X(x)$.
Let
\[
 B = (b_{ij})_{1\le i,j\le d} = \frac{1}{\sqrt{2 C''(0)}}\left(\nabla^2 X - C'(0) X I_d\right)
\]
or
\[
 \nabla^2 X = \sqrt{2 C''(0)} \left(B - \sqrt{\frac{C'(0)^2}{2 C''(0)}} X I_d\right).
\]
Then, $X$, $\nabla X$, $B$ are mutually independent.
Under the conditional distribution given $X=\nu$,
the elements of $B=(b_{ij})$ are jointly Gaussian with
$\E[b_{ij}]=0$ and
\begin{equation}
\label{GOI}
 \E[b_{ij}b_{kl}] = \frac{1}{2}(\delta_{ik}\delta_{jl} + \delta_{il}\delta_{jk}) + \sigma \delta_{ij}\delta_{kl},
\end{equation}
where
\begin{equation}
\label{sigma}
 \sigma = \frac{C''(0)-C'(0)^2}{2 C''(0)}
 \in \left[-\frac{1}{d},\frac{1}{2}\right).
\end{equation}
Recall that we defined
\begin{equation}
\label{gamma}
 \gamma = -C'(0) \in (0,\infty).
\end{equation}

Then, the $f_{m,d}(\nu)$ in (\ref{fmd}) is rewritten as 
\[
  f_{m,d}(\nu)
= \left(\frac{\gamma}{2\pi}\right)^{d/2} \Bigl(\frac{1-2\sigma}{2}\Bigr)^{-d/2}
 \phi(\nu) g_{m,d}\biggl(\nu\sqrt{\frac{1-2\sigma}{2}};\sigma\biggr)
\]
where
\begin{equation}
\label{g}
 g_{m,d}(c;\sigma) = (-1)^m \E_\sigma[\det(B-c I_d) \1\{\ind(B-c I_d)=m\}]. 
\end{equation}
(The parameter $\sigma$ determines the distribution of $B$.)
Similarly,
\begin{equation}
\label{FG}
  F_d(\nu;z)
= \left(\frac{\gamma}{2\pi}\right)^{d/2} \Bigl(\frac{1-2\sigma}{2}\Bigr)^{-d/2}
 \phi(\nu) G_d\biggl(\nu\sqrt{\frac{1-2\sigma}{2}};z;\sigma\biggr)
\end{equation}
where
\begin{equation}
\label{G}
 G_d(c;z;\sigma) = \sum_{m=0}^d z^m g_{m,d}(c;\sigma).
\end{equation}

\subsection{GOI and GOE random matrices}
\label{subsec:goi-goe}

The $d\times d$ centered Gaussian random matrix $B$ with covariance structure (\ref{GOI}) is referred to as the Gaussian orthogonally invariant (GOI) random matrix,
and its probability law is denoted by $\mathtt{GOI}_d(\sigma)$ (\cite{Cheng-Schwartzman:2018}).
For independent $B\sim \mathtt{GOI}_d(\sigma)$ and $\xi\sim\mathcal{N}(0,1)$, it holds that
\begin{equation}
\label{semi_group}
 B + \sqrt{\alpha} \xi I_d \mathop{=}^d B'\sim \mathtt{GOI}_d(\sigma+\alpha)\quad (\alpha\ge 0),
\end{equation}
where $\mathop{=}^d$ denotes equality in distributions.

The special case $\mathtt{GOI}_d(0)$ is referred to as the Gaussian orthogonal ensemble (GOE) and is denoted by $\mathtt{GOE}_d$.
The GOE matrix $B=(b_{ij})\sim \mathtt{GOE}_d$ is a symmetric random matrix such that
\[
 b_{ii}\sim \mathcal{N}(0,1), \quad b_{ij}=b_{ji}\sim \mathcal{N}(0,1/2) \ \ (i\ne j) \ \ \mbox{(independently)}. 
\]
The joint density of $B$ is proportional to $e^{-\tr B^2/2}$.
The density function of the ordered eigenvalues $\lambda_1<\cdots<\lambda_d$ of $B$ is
\begin{equation}
\label{pdf_eigen}
\frac{1}{c_d}\prod_{i>j}(\lambda_i-\lambda_j)e^{-\sum\lambda_i^2/2}, \qquad c_d = 2^{d/2}\prod_{i=1}^d\Gamma(i/2).
\end{equation}

For $B\sim\mathtt{GOE}_d$, let
\begin{equation*}
 \bar B = (\bar b_{ij}) = B-\frac{1}{d}\tr(B)I_d.
\end{equation*}
Then, $\bar B$ is distributed as $\mathtt{GOI}_d(-1/d)$ and is independent of $\tr(B)\sim\mathcal{N}(0,d)$.

\subsection{Pfaffian}
\label{subsec:pffafian}

Here we summarize the pfaffian, which is used to describe our main results in Section \ref{sec:main}.
For an $n\times n$ ($n$:even) skew-symmetric matrix $A=(a_{ij})$, the pfaffian of $A$ is defined to be
\[
 \pf A = \sum_{\pi}\sgn(\pi) a_{\pi(1),\pi(2)} \cdots a_{\pi(n-1),\pi(n)},
\]
where $\pi$ runs over all pairings of $\{1,\ldots,n\}$.
Recall that the pairing $(\pi(1),\ldots,\pi(n))$ is a permutation of $(1,\ldots,n)$ satisfying
\[
 \pi(1)<\pi(2),\ldots,\pi(n-1)<\pi(n) \quad\mbox{and}\quad \pi(1)<\pi(3)<\cdots\pi(n-1).
\]
The signature $\sgn(\pi)$ is defined as the signature of the permutation $\pi$.

For an $n\times n$ matrix $B$, the following identity holds:
\[
 \pf(B^\top A B) = \det(B)\pf(A),
\]
which enables Gaussian elimination by taking $B$ to be an upper or lower triangular matrix.
For example, for a skew-symmetrix matrix
\[
 A = \begin{pmatrix} A_{11} & A_{12} \\ -A_{12}^\top & A_{22} \end{pmatrix}
 \quad \mbox{such that}\ \ \pf(A_{11})\ne 0,
\]
 we have
\begin{align*}
\pf(A) =
\pf\begin{pmatrix} I & 0 \\ A_{12}^\top A_{11}^{-1} & I \end{pmatrix}
\begin{pmatrix} A_{11} & A_{12} \\ -A_{12}^\top & A_{22} \end{pmatrix}
\begin{pmatrix} I & -A_{11}^{-1}A_{12} \\ 0 & I \end{pmatrix} =&
\pf\begin{pmatrix} A_{11} & 0 \\ 0 & A_{22\cdot 1} \end{pmatrix} \\
=&
\pf(A_{11})\pf(A_{22\cdot 1}),
\end{align*}
where $A_{22\cdot 1}=A_{22}+A_{12}^\top A_{11}^{-1}A_{12}$.
Another simple implication is
\begin{equation*}
\pf(B A B)=\left(\prod b_i\right)\pf(A)\quad\mbox{for }B=\diag(b_i).
\end{equation*}

For an $n\times n$ skew-symmetric invertible matrix $A$, and $n\times 1$ column vectors $b$ and $c$,
\begin{equation}
\label{2rank} 
 \pf(A+b c^\top-c b^\top) = \pf(A)(1-b^\top A^{-1}c).
\end{equation}

The pfaffian can be defined without using the pairing.
Let $V$ be a linear space with basis $e_1,e_2,\ldots$,
and consider the exterior algebra $\bigwedge(V)$ with the product
$e_i\wedge e_j=-e_j\wedge e_i$. 
Let
\[
 \omega_A = \sum_{1\le i<j\le n} a_{ij} e_i\wedge e_j. 
\]
Then,
\[
 \frac{1}{(n/2)!} (\omega_A)^{(n/2)} = \pf(A) e_1\wedge e_2 \wedge \cdots \wedge e_n.
\]

\section{Main results}
\label{sec:main}

The goal of this section is to derive $G_d(c;z;\sigma)$ for all $\sigma$.
This immediately yields a formula for $F_d(\nu;z)$.
Our derivation consists of two steps.

In Section \ref{subsec:de_Bruijn},
we prepare a variant of de Bruijn's theorem \cite{deBruijn:1955},
and evaluate $G_d(c;z;0)$ (i.e., the case where $B\sim\mathtt{GOE}_d$) in Theorem \ref{thm:G_goe}.
In Section \ref{subsec:weierstrass}, we establish the inversion of the Weierstrass transform, and demonstrate that $G_d(c;z;\sigma)$ for any $\sigma$ is expressed as the Weierstrass transform, or the inverse of the Weierstrass transform of $G_d(c;z;0)$.
The resulting formulas are summarized in Theorem \ref{thm:main}.
In Section \ref{subsec:dunnett}, we show that this approach is useful for probability calculations in multiple comparisons.

\subsection{A variant of de Bruijn's theorem}
\label{subsec:de_Bruijn}

The formula we use is summarized as follows.

\begin{lemma}[A variant of de Bruijn's theorem]
\label{lem:de_Bruijn}

Let $h(\cdot)$ be a function such that
$\int_{-\infty}^\infty (1+|\lambda|^d) |h(\lambda)|\,\dd\lambda<\infty$.
Then,
\begin{equation}
\label{de_Bruijn}
\begin{aligned}
& \int_{-\infty<\lambda_1<\cdots<\lambda_d<\infty}
 \prod_{1\le i<j\le d}(\lambda_j-\lambda_i)\prod_{1\le i\le d}(\lambda_i-c)
 \prod_{i=1}^d h(\lambda_i) \prod_{i=1}^d \dd\lambda_i \\
&\qquad = \begin{cases}\displaystyle
\pf\begin{pmatrix} V & p(c) \\
 - p(c)^\top & 0 \end{pmatrix} & (d: \mbox{even}), \\[4mm]
  \displaystyle
\pf\begin{pmatrix} V & W & p(c) \\
 -W^\top & 0 & 1 \\
 - p(c)^\top & -1 & 0 \end{pmatrix} -\pf(V) & (d: \mbox{odd}),
\end{cases}
\end{aligned}
\end{equation}
where $p(c)=(p_0(c),p_1(c),\ldots,p_d(c))^\top$
with $p_i(\cdot)$ a monic polynomial of degree $i$,
$V=(v_{ij})_{1\le i,j\le d+1}$ with
\[
 v_{ij} = \int_{-\infty<\lambda_1<\lambda_2<\infty}
 \det\begin{pmatrix}
  p_{i-1}(\lambda_1) & p_{j-1}(\lambda_1) \\
  p_{i-1}(\lambda_2) & p_{j-1}(\lambda_2)
  \end{pmatrix}
  h(\lambda_1) h(\lambda_2)\,\dd\lambda_1\,\dd\lambda_2,
\]
and
$W=(w_1,\ldots,w_{d+1})^\top$ with
\[
 w_i = \int_{-\infty}^\infty p_{i-1}(\lambda) h(\lambda)\,\dd\lambda.
\]
\end{lemma}

\begin{proof}
Suppose that the support $D=\{x\in\R \,|\, h(x)\ne 0\}$ has a positive Lebesgue measure.
(Otherwise (\ref{de_Bruijn}) holds trivially.)
Let $c\in D$.
We apply Theorem \ref{thm:incomplete} with $\phi_i(x)=p_{i-1}(x)h(x)$ and $x_{d+1}=c$.
Since the linkage factor is represented as the Vandermonde's determinant
\[
 \prod_{1\le i<j\le d+1}(x_j-x_i) = \det(x_j^{i-1})_{1\le i,j\le d+1} = \det(p_{i-1}(x_j))_{1\le i,j\le d+1},
\]
the integrand of (\ref{I}) is
\begin{align*}
 \det\bigl(\phi_i(x_j)\bigr)_{1\le i,j\le d+1}
 &= \prod_{1\le i<j\le d+1}(x_j-x_i) \prod_{i=1}^{d+1} h(x_i) \\
 &= (-1)^d \prod_{1\le i<j\le d}(x_j-x_i)\prod_{i=1}^d (x_i-c) \prod_{i=1}^{d} h(x_i) \times h(c),
\end{align*}
which is proportional to the integrand in (\ref{de_Bruijn}).
By dividing both sides of (\ref{I}) by $h(c)$, we obtain the formula (\ref{de_Bruijn}) for $c\in D$.

Since both sides of (\ref{de_Bruijn}) are polynomials in $c$, (\ref{de_Bruijn}) holds for all $c\in\R$.
\end{proof}

Let $B$ be a $d\times d$ GOE matrix (i.e., $\mathtt{GOI}(\sigma)$ with $\sigma=0$), and let $\lambda_1<\cdots<\lambda_d$ be its ordered eigenvalues.
Since
\begin{align*}
 \prod_{i=1}^d \bigl\{ \1(\lambda_i \ge c) -z \1(\lambda_i < c) \bigr\}
 &= \sum_{m=0}^d (-z)^m \1\{\#\{i: \lambda_i < c \}=m\} \\
 &= \sum_{m=0}^d (-z)^m \1\{\ind(B-cI)=m\},
\end{align*}
we have
\begin{equation}
\label{Gdcz}
 G_d(c;z;0) = \E\left[\prod_{i=1}^d (\lambda_i-c) \prod_{i=1}^d \bigl\{ \1(\lambda_i \ge c) -z\1(\lambda_i < c) \bigr\} \right]
\end{equation}
where the expectation is taken with respect to the density function in (\ref{pdf_eigen}).

By applying Lemma \ref{lem:de_Bruijn} with
\[
 h(\lambda) = \bigl\{\1(\lambda \ge c) -z\1(\lambda < c) \bigr\} e^{-\lambda^2/2},
\]
we obtain the explicit formula for $G_d(c;z;0)$.

\begin{theorem}
\label{thm:G_goe}
\begin{align*}
& G_d(c;z;0) = \\
& \frac{1}{c_d} \begin{cases}
 \displaystyle
 \pf\begin{pmatrix} V^+(c) + z^2 V^-(c) -z \bigl(w^-(c)w^+(c)^\top - w^+(c)w^-(c)^\top\bigr)  & p(c) \\ -p(c)^\top & 0 \end{pmatrix} & \hspace*{-4em}(d {\rm :even}), \\[3ex]
 \displaystyle
 \pf\begin{pmatrix}
  V^+(c) + z^2 V^-(c) -z \bigl(w^-(c)w^+(c)^\top - w^+(c)w^-(c)^\top\bigr) & w^+(c)-z w^-(c) & p(c) \\
 -\bigl(w^+(c) -z w^-(c)\bigr)^\top & 0  & 1 \\
 -p(c)^\top                         & -1 & 0
\end{pmatrix} \\
\quad - \displaystyle
 \pf\begin{pmatrix}
  V^+(c) + z^2 V^-(c) -z \bigl(w^-(c)w^+(c)^\top - w^+(c)w^-(c)^\top\bigr)
\end{pmatrix}
& \hspace*{-4em}(d {\rm :odd})
\end{cases}
\end{align*}
where
$p(c)=(p_0(c),p_1(c),\ldots,p_d(c))^\top$ with $p_i(\cdot)$ a monic polynomial of degree $i$,
$V^\pm(c)=(v^\pm_{ij}(c))_{1\le i,j\le d+1}$ with
\begin{align*}
v^+_{ij}(c) =& \int_{c<\lambda_1<\lambda_2}
 \det\begin{pmatrix}
  p_{i-1}(\lambda_1) & p_{j-1}(\lambda_1) \\
  p_{i-1}(\lambda_2) & p_{j-1}(\lambda_2)
  \end{pmatrix}
  e^{-(\lambda_1^2+\lambda_2^2)/2}\,\dd\lambda_1\,\dd\lambda_2, \\
v^-_{ij}(c) =& \int_{\lambda_1<\lambda_2<c}
 \det\begin{pmatrix}
  p_{i-1}(\lambda_1) & p_{j-1}(\lambda_1) \\
  p_{i-1}(\lambda_2) & p_{j-1}(\lambda_2)
  \end{pmatrix}
  e^{-(\lambda_1^2+\lambda_2^2)/2}\,\dd\lambda_1\,\dd\lambda_2,
\end{align*}
$w^\pm(c)=(w^\pm_1(c),\ldots,w^\pm_{d+1}(c))^\top$
with
\[
 w^+_i(c) = \int_{c}^\infty p_{i-1}(\lambda) e^{-\lambda^2/2}\,\dd\lambda, \qquad
 w^-_i(c) = \int^{c}_{-\infty} p_{i-1}(\lambda) e^{-\lambda^2/2}\,\dd\lambda.
\]
The constant $c_d$ is given in (\ref{pdf_eigen}).
\end{theorem}

\begin{proof}
The $(i,j)$th element of $V$ in Lemma \ref{lem:de_Bruijn} is
\begin{equation}
\label{vij}
\begin{aligned}
 v_{ij} &= v_{ij}(c;z) \\
 &= \left(\int_{c<\lambda_1<\lambda_2} - z\int_{\lambda_1<c<\lambda_2} + z^2 \int_{\lambda_1<\lambda_2<c}\right)
 \det\begin{pmatrix}
  p_{i-1}(\lambda_1) & p_{j-1}(\lambda_1) \\
  p_{i-1}(\lambda_2) & p_{j-1}(\lambda_2)
  \end{pmatrix}
  e^{-(\lambda_1^2+\lambda_2^2)/2}\,\dd\lambda_1\,\dd\lambda_2 \\[1ex]
&= v^+_{ij}(c) - z (w^-_{i}(c) w^+_{j}(c)-w^+_{i}(c)w^-_{j}(c)) + z^2 v^-_{ij}(c),
\end{aligned}
\end{equation}
and the $i$th element of $W$ in Lemma \ref{lem:de_Bruijn} is
\begin{align*}
w_i = w_i(c;z) = \left(\int_{c}^\infty -z \int_{-\infty}^c \right)
 p_{i-1}(\lambda) e^{-\lambda^2/2}\,\dd\lambda
= w^+_i(c) -z w^-_i(c).
\end{align*}
\end{proof}

In this paper, we present the formulas using the complementary error function
\[
 \erfc(x) = \frac{2}{\sqrt{\pi}}\int_x^\infty e^{-t^2}\,\dd t = 1-\erf(x), 
\]
where $\erf$ is the error function.
This is an entire function on $\C$ and satisfies
\[
 \erfc(x) + \erfc(-x) = 2, \qquad \erfc(0)=1, \qquad \erfc(\pm\infty) = 1\mp 1.
\]
The upper tail probability of a standard Gaussian random variable $\xi\sim\mathcal{N}(0,1)$ is
\[
 \Pr(\xi > x) = \frac{1}{2}\erfc\bigl(x/\sqrt{2}\bigr).
\]

\begin{lemma}
Suppose that $p_i(c)=c^{i}$ and $p(c)=(1,c,\ldots,c^d)^\top$ in Theorem \ref{thm:G_goe}.
Then
\begin{align*}
& v^+_{ij}(c) = \int_{c<\lambda_1<\lambda_2}
 \det\begin{pmatrix}
  \lambda_1^{i-1} & \lambda_1^{j-1} \\
  \lambda_2^{i-1} & \lambda_2^{j-1}
  \end{pmatrix}
  e^{-(\lambda_1^2+\lambda_2^2)/2}\,\dd\lambda_1\,\dd\lambda_2,
\end{align*}
and
\[
 w^+_i(c) = \int_{c}^\infty \lambda^{i-1} e^{-\lambda^2/2}\,\dd\lambda
\]
are evaluated according to the recurrence formulas below.

For $1\le i<j$,
\begin{align*}
 v^+_{i,j}(c)
=& -c^{j-2} e^{-c^2/2} w^+_{i}(c) + \frac{1}{2^{(i+j-4)/2}} w^+_{i+j-2}(\sqrt{2}c) \\
&+ \begin{cases}
(j-2)v^+_{i,j-2}(c) & (j\ge i+3), \\
0 & (j=i+2 \mbox{ or } (i,j)=(1,2)), \\
-(j-2)v^+_{j-2,i}(c) & (j=i+1 \mbox{ and }(i,j)\ne (1,2)),
\end{cases}
\end{align*}
and for $1\le i$,
\begin{align*}
w^+_{i}(c)
=& \begin{cases}
\displaystyle
c^{i-2} e^{-c^2/2} + (i-2) w^+_{i-2}(c) & (i\ge 3), \\
\displaystyle
e^{-c^2/2} & (i=2), \\
\displaystyle
\int_c^{\infty} e^{-\lambda^2/2}\,\dd\lambda = \sqrt{\frac{\pi}{2}}\erfc\bigl(c/\sqrt{2}\bigr) & (i=1).
\end{cases}
\end{align*}
$v^-_{ij}(c)$ and $w^-_i(c)$ are obtained as
\[
 v^-_{ij}(c)=(-1)^{i+j-1}v^+_{ij}(-c), \qquad
 w^-_i(c)=(-1)^{i-1}w^+_i(-c).
\]
\end{lemma}

\begin{proof}
Suppose that $1\le i<j$.
By integration by parts,
\begin{align*}
 v^+_{i,j}(c)
=& \int_{c}^{\infty} \lambda_1^{i-1} e^{-\lambda_1^2/2}\,\dd\lambda_1 \int_{\lambda_1}^{\infty} \lambda_2^{j-1} e^{-\lambda_2^2/2}\,\dd\lambda_2 - \int_c^{\infty} \lambda_2^{i-1} e^{-\lambda_2^2/2}\,\dd\lambda_2 \int_{c}^{\lambda_2} \lambda_1^{j-1} e^{-\lambda_1^2/2}\,\dd\lambda_1 \\
=& \int_{c}^{\infty} \lambda_1^{i-1} e^{-\lambda_1^2/2}\,\dd\lambda_1 \left[ -\lambda_2^{j-2} e^{-\lambda_2^2/2} \bigg|_{\lambda_1}^\infty + (j-2)\int_{\lambda_1}^\infty \lambda_2^{j-3} e^{-\lambda_2^2/2}\,\dd\lambda_2 \right] \\
& - \int_c^{\infty} \lambda_2^{i-1} e^{-\lambda_2^2/2}\,\dd\lambda_2 \left[ -\lambda_1^{j-2} e^{-\lambda_1^2/2} \bigg|_{c}^{\lambda_2} + (j-2)\int_{c}^{\lambda_2} \lambda_1^{j-3} e^{-\lambda_1^2/2}\,\dd\lambda_1 \right] \\
=& -c^{j-2} e^{-c^2/2} w^+_{i}(c) + 2\int_c^{\infty} \lambda^{i+j-3} e^{-\lambda^2}\,\dd\lambda + (j-2) v^+_{i,j-2}(c) \\
=& -c^{j-2} e^{-c^2/2} w^+_{i}(c) + \frac{1}{2^{(i+j-4)/2}} w^+_{i+j-2}(\sqrt{2}c) \\[0.5ex]
&+ \begin{cases}
(j-2)v^+_{i,j-2}(c) & (j\ge i+3), \\[0.5ex]
0 & (j=i+2 \mbox{ or }j=2), \\[0.5ex]
-(j-2)v^+_{j-2,i}(c) & (j<i+2),
\end{cases}
\end{align*}
and
\begin{align*}
w^+_{i}(c)
=& -\lambda^{i-2} e^{-\lambda^2/2}\bigg|_{c}^{\infty} + (i-2) \int_{c}^{\infty} \lambda^{i-3} e^{-\lambda^2/2}\,\dd\lambda \\[0.5ex]
=& \begin{cases}
\displaystyle
c^{i-2} e^{-c^2/2} + (i-2) w^+_{i-2}(c) & (i\ge 2), \\[1ex]
\displaystyle
\int_c^{\infty} e^{-\lambda^2/2}\,\dd\lambda = \sqrt{\frac{\pi}{2}}\erfc\bigl(c/\sqrt{2}\bigr) & (i=1).
\end{cases}
\end{align*}
\end{proof}

\begin{remark}
When $z=-1$, $G_d(c;z;0)$ in (\ref{Gdcz}) becomes
\begin{equation}
\label{z=-1}
\begin{aligned}
 G_d(c;-1;0)
 =& \E\left[\prod_{i=1}^d(\lambda_i-c)\right] = \E[\det(B-c I_d)]
 \quad (B\sim\mathtt{GOE}_d) \\
 =& \frac{(-1)^d}{2^{d/2}} H_d(\sqrt{2}c)
\end{aligned}
\end{equation}
(see, e.g., \cite{Forrester:2013}).

When $z=1$, $G_d(c;z;0)$ in (\ref{Gdcz}) becomes
\begin{equation}
\label{z=1}
\begin{aligned}
 G_d\bigl(c;1;0\bigr) =& 
 \sqrt{\frac{1}{2\pi}}\,\left(\frac{d-1}{2}\right)!\Biggl[
 2\,e^{-c^2/2}
 \sum_{k=0}^{d}\frac{H_{k}^2(\sqrt{2}c)}{k!} \\
&\qquad\qquad\qquad\quad +\frac{1}{2^{1/2}d!}\,H_d(\sqrt{2}c)\,
\int_{-\infty}^{\infty} e^{-u^2/2}\,H_{d+1}(\sqrt{2}u)\,\sgn(c-u)\,\dd u
\Biggr] \\
& + \begin{cases}
\displaystyle
 \frac{1}{2^{d/2}}H_d(\sqrt{2}c) & (\mbox{$d$ is even}), \\[1ex]
 0 & (\mbox{$d$ is odd}).
\end{cases}
\end{aligned}
\end{equation}

The result for odd $d$ is given by \citet[Eq.\,(10)]{Fyodorov:2004}.
A sketch of the proof of (\ref{z=-1}) and (\ref{z=1}) when $d$ is even is  given in Appendix \ref{subsec:pm1}.
\end{remark}

\subsection{Weierstrass transform and its inversion}
\label{subsec:weierstrass}

In what follows, let $\xi$ and $\xi'$ be independent standard Gaussian random variables.

\begin{lemma}
\label{lem:negative_var}
Let $f(x)$ be a real-valued function on $\R$ such that $|f(x)|$ grows at most of polynomial order in $|x|$ as $x\to\pm\infty$.
Let
\begin{equation*}
 h(x;\alpha) = \E[f(x+\sqrt{\alpha}\xi)]
 = \int_{-\infty}^{\infty} f(\xi) \phi\biggl(\frac{\xi-x}{\sqrt{\alpha}}\biggr) \,\frac{\dd\xi}{\sqrt{\alpha}}, \quad x\in\R,\ \ \alpha>0,
\end{equation*}
where $\phi$ is the density function of $\mathcal{N}(0,1)$ defined in (\ref{phi-H}).
Then, for each $\alpha>0$, it extends to an entire function $h(z;\alpha)$, $z\in\C$.

For $\alpha, \beta>0$, it holds that
\begin{equation}
\label{equality}
 h(z;\alpha+\beta) = \E[h(z+\sqrt{\beta}\xi;\alpha)] \quad\mbox{for all } z\in\C.
\end{equation}

Moreover, for $\alpha, \beta>0$,
\begin{equation}
\label{equality1}
 h(z;\alpha) = \E[h(z+\sqrt{-1}\sqrt{\beta}\xi;\alpha+\beta)] \quad\mbox{for all } z\in\C
\end{equation}
holds.
\end{lemma}

\begin{proof}
Recall that
\[
 h(x;\alpha) = \int_{-\infty}^{\infty} f(\xi) \phi\biggl(\frac{\xi-x}{\sqrt{\alpha}}\biggr) \,\frac{\dd\xi}{\sqrt{\alpha}}
\]
is called the Weierstrass transform of $f$.
By \citet[Theorem 13.3]{Hirschman-Widder:1965}, $h(z;\alpha)$ exists as an entire function.

From the property (\ref{semi_group}), the first identity (\ref{equality}) holds for $z\in\R$.
(That is, $h(\cdot;\alpha)$ is the Weierstrass transform of $h(\cdot;\alpha')$, $\alpha'<\alpha$.)
By the identity theorem for holomorphic functions (\citet[Theorem 1.2 of Chapter 3]{Stein-Shakarchi:2003}), (\ref{equality}) holds for all $z\in\C$.

Substituting $z:=z+\sqrt{-1}\sqrt{\beta}\xi'$ into (\ref{equality}) and taking expectation with respect to $\xi'\sim\mathcal{N}(0,1)$, we have
\begin{align*}
\E[h(z+\sqrt{-1}\sqrt{\beta}\xi';\alpha+\beta)]
=& \int_{-\infty}^{\infty} \left[ \int_{-\infty}^{\infty} f(\xi) \phi\biggl(\frac{\xi-(z+\sqrt{-1}\sqrt{\beta}\xi')}{\sqrt{\alpha+\beta}}\biggr) \,\frac{\dd\xi}{\sqrt{\alpha+\beta}} \right] \phi(\xi') \dd\xi' \\
=& \int_{-\infty}^{\infty} f(\xi) \left[ \int_{-\infty}^{\infty} \phi\biggl(\frac{\xi-(z+\sqrt{-1}\sqrt{\beta}\xi')}{\sqrt{\alpha+\beta}}\biggr) \phi(\xi') \dd\xi' \right]\,\frac{\dd\xi}{\sqrt{\alpha+\beta}} \\
=& \int_{-\infty}^{\infty} f(\xi) \phi\biggl(\frac{\xi-z}{\sqrt{\alpha}}\biggr) \frac{\dd\xi}{\sqrt{\alpha}} = h(z;\alpha),
\end{align*}
that is, (\ref{equality1}) holds for all $z\in\C$.
The interchange of the integrals is justified by absolute integrability.
\end{proof}

\begin{remark}
When $f(x)$ is defined on $\C$, (\ref{equality1}) reads
\[
 \E[f(x+\sqrt{\alpha}\xi)] = \E[f(x+\sqrt{\alpha+\beta}\xi+\sqrt{-1}\sqrt{\beta}\xi')]. 
\]
The random variable $\sqrt{\alpha+\beta}\xi+\sqrt{-1}\sqrt{\beta}\xi'$, $\xi,\xi'\sim \mathcal{N}(0,1)$ i.i.d\ behaves like an $\mathcal{N}(0,\alpha)$.
\end{remark}

We now return to our problem.
We apply Lemma \ref{lem:negative_var} to the function $g_{m,d}(c;\sigma)$ defined in (\ref{g}).
Recall that $\sigma = -1/d$ is the possible minimum value of $\sigma$.
Note first that
\begin{align*}
 |g_{m,d}(c;-1/d)| \le& \E[|\det(B-cI)|], \quad B\sim\mathtt{GOI}(-1/d), \\
 \le& (\mbox{a polynomial in $|c|$ of degree $d$}).
\end{align*}
For any $\sigma\in(-1/d,1/2)$, $g_{m,d}(\cdot;\sigma)$ is represented as the Weierstrass transform of $g_{m,d}(\cdot;-1/d)$:
\[
 g_{m,d}(c;\sigma) = \E[g_{m,d}(c+\sqrt{\sigma+(1/d)}\,\xi;-1/d)].
\]
Therefore, $g_{m,d}(\cdot;\sigma)$ for $\sigma\in(-1/d,1/2)$ is an entire function, and is represented by $g_{m,d}(\cdot;0)$ as
\[
 g_{m,d}(c;\sigma) = \begin{cases}
 \displaystyle
 \E[g_{m,d}(c+\sqrt{\sigma}\xi;0)] & (\sigma>0)\ \ \mbox{by (\ref{equality})}, \\
 \displaystyle
 g_{m,d}(c;0) & (\sigma=0), \\
 \displaystyle
 \E[g_{m,d}(c+\sqrt{-1}\sqrt{-\sigma}\xi;0)] & (\sigma<0)\ \ \mbox{by (\ref{equality1})},
 \end{cases}
\]
or simply,
\begin{equation*}
 g_{m,d}(c;\sigma) = \E[g_{m,d}(c+\sqrt{\sigma}\xi;0)] \quad \mbox{for }\sigma\in(-1/d,1/2).
\end{equation*}
Similarly, for $G_d(c;z;\sigma)$ in (\ref{G}) we immediately have
\begin{equation}
\label{G_conv} 
 G_d(c;z;\sigma) = \E[G_{d}(c+\sqrt{\sigma}\xi;z;0)] \quad \mbox{for }\sigma\in(-1/d,1/2).
\end{equation}

By combining (\ref{FG}) and (\ref{G_conv}),
we obtain the main theorem:

\begin{theorem}
\label{thm:main}
The generating function $F_d(\nu;z)$ of the expected number density $f_{m,d}(\nu)$ is
\begin{equation}
\label{main}
  F_d(\nu;z)
= \Bigl(\frac{\gamma}{2\pi}\Bigr)^{d/2} \Bigl(\frac{1-2\sigma}{2}\Bigr)^{-d/2}
 \phi(\nu)
 \int_{-\infty}^\infty
 G_d\biggl(\nu\sqrt{\frac{1-2\sigma}{2}}+\sqrt{\sigma}\xi;z;0\biggr)
 \phi(\xi) \,\dd\xi,
\end{equation}
where $\gamma>0$ and $\sigma\in(-1/d,1/2)$ are defined in (\ref{sigma}), and $G_d(\cdot;z;0)$ is given in Theorem \ref{thm:G_goe}. 
\end{theorem}

\begin{remark}
In (\ref{main}), when \(\sigma<0\),
we take the principal branch $\sqrt{\sigma}=\sqrt{-1}\sqrt{-\sigma}$.
The opposite branch gives the same value, since \(\xi\stackrel{d}{=}-\xi\).
Although the integrand in (\ref{main}) is then complex-valued, the integral (\ref{main}) is real for real $\nu$ and $z$.
\end{remark}

\begin{remark}
Theorem \ref{thm:main} does not cover the boundary case $\sigma=-1/d$,
although this is in the feasible parameter space (\ref{sigma}).
Since $g_{m,d}(x;\sigma)$ in (\ref{g}) is continuous in $\sigma$ at $-1/d$,
we have
\[
 G_d(c;z;-1/d) = \lim_{\sigma\downarrow -1/d}
 \int_{-\infty}^\infty G_d\bigl(c+\sqrt{\sigma}\xi;z;0\bigr) \phi(\xi) \,\dd\xi.
\]

For example, when $d=1,2$, direct calculation using $g_{m,d}$ in (\ref{g}) yields
\begin{equation}
\label{G12boundary}
\begin{aligned}
 G_1(x;z;-1) =&
 -x\1(x<0) + z x\1(x>0), \\
 G_2(x;z;-1/2) =& 
  (e^{-x^2}+x^2-1)\1(x<0)
  + z e^{-x^2}
  + z^2 (e^{-x^2}+x^2-1)\1(x>0),
\end{aligned}
\end{equation}
 which are shown to equal the limits $\lim_{\sigma\downarrow -1/d}G_d(x;z;\sigma)$, $d=1,2$, respectively (see (\ref{F12boundary}) in Appendix \ref{subsec:d=1,2}).
\end{remark}

\begin{remark}
Recall that $G_d(c;z;0)$ with $z=-1$ was 
\[
 G_d(c;-1;0)  = (-1)^d 2^{-d/2} H_d(\sqrt{2}c)
\]
(see (\ref{z=-1})).
Then, we have the known formula
\begin{equation}
\label{known_formula}
\begin{aligned}
  F_d(\nu;-1)
=& (-1)^d \Bigl(\frac{\gamma}{2\pi}\Bigr)^{d/2} (1-2\sigma)^{-d/2}
 \phi(\nu) \int_{-\infty}^\infty H_d\Bigl(\nu\sqrt{1-2\sigma}+\sqrt{2\sigma}\xi\Bigr) \phi(\xi) \,\dd\xi \\
=& (-1)^d \Bigl(\frac{\gamma}{2\pi}\Bigr)^{d/2}\phi(\nu) H_{d}(\nu),
\end{aligned}
\end{equation}
which is independent of $\sigma$
(e.g., \citet{Adler:1981}, \citet{Tomita:1986}).
The last equality in (\ref{known_formula}) is proved by a generating function approach:
\begin{align*}
\sum_{d=0}^\infty & \frac{1}{d!}w^d (1-2\sigma)^{-d/2} \int_{-\infty}^\infty H_d\Bigl(\nu\sqrt{1-2\sigma}+\sqrt{2\sigma}\xi\Bigr) \phi(\xi) \,\dd\xi \nonumber \\
&= \int_{-\infty}^\infty \exp\biggl(\frac{w}{\sqrt{1-2\sigma}} \Bigl(\nu\sqrt{1-2\sigma}+\sqrt{2\sigma}\xi\Bigr)-\frac{1}{2}\frac{w^2}{1-2\sigma}\biggr) \phi(\xi)\,\dd\xi \\
&= \exp\biggl(\nu w + \frac{1}{2}\frac{2\sigma w^2}{1-2\sigma} -\frac{1}{2}\frac{w^2}{1-2\sigma}\biggr)
= \exp\biggl(\nu w -\frac{1}{2}w^2\biggr) \\
&= \sum_{d=0}^\infty \frac{1}{d!}w^d H_d(\nu).
\end{align*}
\end{remark}

\subsection{$G_d(x;z;0)$ for $d=1,2,3$}

The functions $G_d(x;z;0)$ for $d=1,2,3$ are displayed below.

$d=1$:
\[
 G_1(x;z;0) =
 \Biggl(\frac{e^{-\frac{x^2}{2}}}{\sqrt{2\pi}} -\frac{x\,\erfc\bigl(\frac{x}{\sqrt{2}}\bigr)}{2}\Biggr)
+ \Biggl(\frac{e^{-\frac{x^2}{2}}}{\sqrt{2\pi}} +\frac{x\,\erfc\bigl(-\frac{x}{\sqrt{2}}\bigr)}{2}\Biggr)z
\]

$d=2$:
\begin{align*}
 G_2(x;z;0)
=&
\Biggl(
-\frac{e^{-x^2} x}{2\sqrt{\pi}}
+\frac{e^{-\frac{x^2}{2}} \erfc\bigl(\frac{x}{\sqrt{2}}\bigr)}{2\sqrt{2}}
+\frac{(2 x^2-1) \erfc(x)}{4}
\Biggr)
 + \frac{e^{-\frac{x^2}{2}}}{\sqrt{2}} z \\
&+\Biggl(
\frac{e^{-x^2} x}{2\sqrt{\pi}}
+\frac{e^{-\frac{x^2}{2}} \erfc\bigl(-\frac{x}{\sqrt{2}}\bigr)}{2\sqrt{2}}
+\frac{(2 x^2-1) \erfc(-x)}{4}
\Biggr)z^2
\end{align*}

$d=3$:
\begin{align*}
G_3(x;z;0)
=&
\Biggl(
-\frac{3 e^{-\frac{3 x^2}{2}} x}{2\sqrt{2}\pi}
+\frac{(x^2-2) e^{-x^2} \erfc\bigl(\frac{x}{\sqrt{2}}\bigr)}{4\sqrt{\pi}} \\
&\quad
+\frac{3 (2 x^2+1) e^{-\frac{x^2}{2}} \erfc(x)}{4\sqrt{2\pi}}
-\frac{(2 x^3-3 x) \erfc(x) \erfc\bigl(\frac{x}{\sqrt{2}}\bigr)}{8}
\Biggr) \\
&
+ \Biggl(
-\frac{3 e^{-\frac{3 x^2}{2}} x}{2\sqrt{2}\pi}
-\frac{(x^2-2) e^{-x^2} \erfc\bigl(-\frac{x}{\sqrt{2}}\bigr)}{4\sqrt{\pi}} \\
&\quad
+\frac{3 (2 x^2+1)e^{-\frac{x^2}{2}} \erfc(x)}{4\sqrt{2\pi}}
+\frac{(2 x^3-3 x) \erfc(x) \erfc\bigl(-\frac{x}{\sqrt{2}}\bigr)}{8}
\Biggr)z \\
&
+\Biggl(
\frac{3 e^{-\frac{3 x^2}{2}} x}{2\sqrt{2}\pi}
-\frac{(x^2-2) e^{-x^2} \erfc\bigl(\frac{x}{\sqrt{2}}\bigr)}{4\sqrt{\pi}} \\
&\quad
+\frac{3(2 x^2+1) e^{-\frac{x^2}{2}} \erfc(-x)}{4\sqrt{2\pi}}
-\frac{(2 x^3-3 x) \erfc(-x) \erfc\bigl(\frac{x}{\sqrt{2}}\bigr)}{8}
\Biggr)z^2 \\
&
+\Biggl(
\frac{3 e^{-\frac{3 x^2}{2}} x}{2\sqrt{2}\pi}
+\frac{(x^2-2) e^{-x^2} \erfc\bigl(-\frac{x}{\sqrt{2}}\bigr)}{4\sqrt{\pi}} \\
&\quad
+\frac{3 (2x^2+1) e^{-\frac{x^2}{2}} \erfc(-x)}{4\sqrt{2\pi}}
+\frac{(2 x^3-3 x) \erfc(-x) \erfc\bigl(-\frac{x}{\sqrt{2}}\bigr)}{8}
\Biggr)z^3
\end{align*}

Except for the case $\sigma=0$,
we need to transform $G_d(\cdot;z;0)$ to $G_d(\cdot;z;\sigma)$ by the Weierstrass transform in Theorem \ref{thm:main}.
For $d\le 2$, this transform can be carried out analytically,
and $F_d(\nu;z)$ is expressed without an integral representation.
For the explicit forms of $F_1(\nu;z)$ and $F_2(\nu;z)$, see Appendix \ref{subsec:d=1,2}.

Figure \ref{fig:number_density} depicts the expected number densities $f_{m,d}(\nu)$ when $d=3$ and $m=0,1,2,3$.
Noting the parameter space (\ref{sigma}) and (\ref{gamma}),
we set the parameters $\gamma=1$ and $\sigma=-0.2,-0.1,\ldots,+0.4$.
When $\sigma<0$, we require numerical packages for the error functions with complex arguments, which are available in standard program packages (e.g., \texttt{Mathematica} or \texttt{Python/mpmath}).

\bigskip
\begin{figure}[h]
\begin{center}
\begin{tabular}{cc}
\scalebox{0.7}{\includegraphics{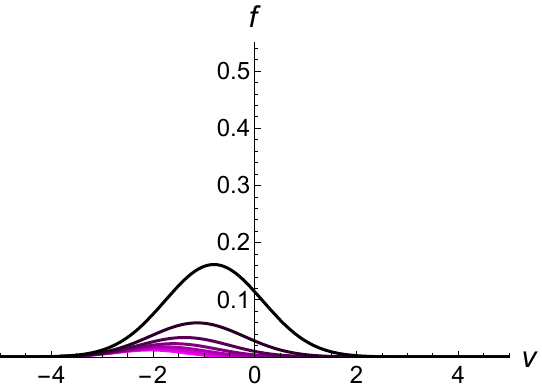}} &
\scalebox{0.7}{\includegraphics{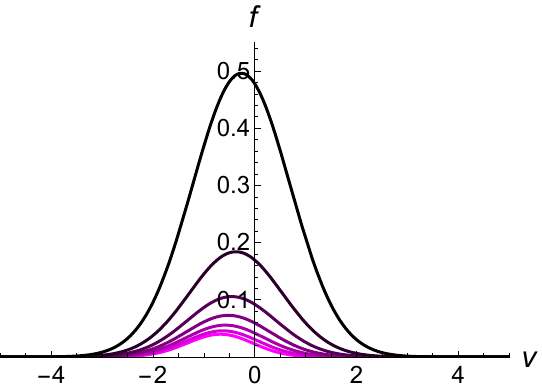}} \\
$\mathrm{index}=0$ & $\mathrm{index}=1$ \\[2mm] 
\scalebox{0.7}{\includegraphics{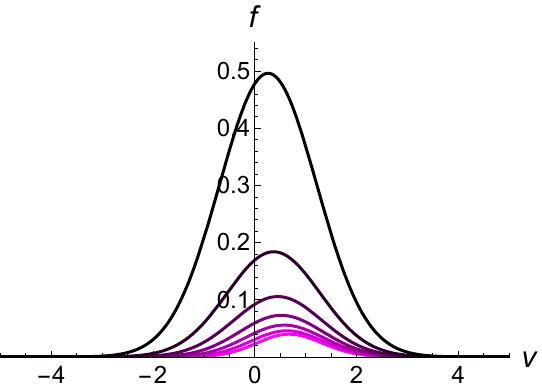}} &
\scalebox{0.7}{\includegraphics{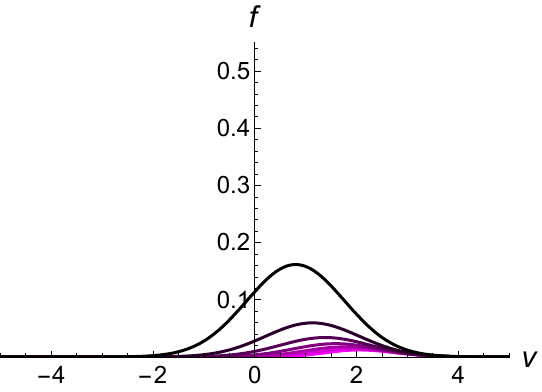}} \\
$\mathrm{index}=2$ & $\mathrm{index}=3$
\end{tabular}
\caption{$f_{m,d}$ for $d=3$, $m\,(\mathrm{index})=0,1,2,3$, $\gamma=1$.\\ \small
(Magenta line: $\sigma=-0.2$, Black line: $\sigma=0.4$, Middle 5 lines: $\sigma=-0.1,0,0.1,0.2,0.3$.)}
\label{fig:number_density}
\end{center}
\end{figure}

\subsection{Dunnett integral}
\label{subsec:dunnett}

In multiple comparisons, the distribution of the maximum of several test statistics is needed to assess the significance.
The typical and simplest case is as follows:

Let $X_1,\ldots,X_d$ be Gaussian random variables with common variance,
but they are correlated.
We assume the simplest model
\[
 (X_1,\ldots,X_d)\sim
\mathcal{N}_d\bigl(0,I_d + \rho \1_d\1_d^\top\bigr),
\quad -1/d\le \rho <\infty,
\]
where $\mathbf{1}_d=(1,\ldots,1)^\top$ is a $d\times 1$ constant vector, and consider the probability $F_\rho(c)=\Pr\bigl(\max_{i=1,\ldots,d} X_i \le c\bigr)$.
Let $\xi_0,\xi_1,\xi_2,\ldots$ be i.i.d.\ standard Gaussian random variables.
When $\rho>0$, the random variables $(X_1,\ldots,X_d)$ are 
represented as
\[
 X_i = \xi_i-\sqrt{\rho}\,\xi_0, \quad i=1,\ldots,d.
\]
Based on this representation,  we have
\begin{align*}
 F_{\rho}(c)
=& \E[\Pr(\forall i, \xi_i\le c + \sqrt{\rho}\,\xi_0 \,|\, \xi_0)] \\
=& \E\bigl[\Phi(c+\sqrt{\rho}\,\xi_0)^d\bigr] \qquad (\rho\ge 0),
\end{align*}
where $\Phi$ is the cumulative distribution function of $\mathcal{N}(0,1)$
(\cite{Dunnett:1989,Dunnett:1993}).
This method is useful for numerical calculation.

However, this method works only when $\rho\ge 0$. 
In the following, we prove that
\begin{align*}
 F_{\rho}(c) = \E\bigl[\Phi(c+\sqrt{-1}\sqrt{|\rho|}\,\xi_0)^d\bigr] \quad \mbox{for }-1/d<\rho<0.
\end{align*}
Note first that
\[
 F_{-1/d}(c) = \Pr\bigl({\max}_{i=1,\ldots,d}\,(\xi_i-\bar\xi) \le c\bigr), \quad \bar\xi=\sum_{i=1}^d\xi_i/d,
\]
is well-defined.
Let $f(c)=F_{-1/d}(c)$ and $h(c;\alpha)=\E[f(c+\sqrt{\alpha}\xi)]$.
Then, $F_\rho(c)=h(c;1/d-|\rho|)$ and $F_0(c)=h(c;1/d)=\Phi(c)^d$ hold.
Since $|f(c)|$ is bounded, by applying Lemma \ref{lem:negative_var}, we have
\begin{align*}
 F_{\rho}(c)
 =& h(c;1/d-|\rho|) \\
 =& \E\bigl[h(c+\sqrt{-1}\sqrt{|\rho|}\xi;1/d)\bigr] \quad (\mbox{by }(\ref{equality1})) \\
 =& \E\bigl[\Phi(c+\sqrt{-1}\sqrt{|\rho|}\xi)^d\bigr].
\end{align*}

\appendix

\section{}
\label{sec:appendix}

\subsection{Incomplete de Bruijn's integral}
\label{subsec:oncomplete_de_Bruijn}

Here we provide an incomplet de Bruijn's integral formula, which implies Lemma \ref{lem:de_Bruijn} as a special case.
We follow the notation of \citet{deBruijn:1955}.

\begin{theorem}[Incomplete de Bruijn's integral]
\label{thm:incomplete}
Let $\phi_1,\ldots,\phi_{d+1}$ be integrable functions. 
Let
\begin{equation}
\label{I}
    I(x_{d+1}) = \int_{-\infty<x_1<\cdots<x_{d}<\infty}
 \det\bigl(\phi_i(x_j)\bigr)_{1\le i,j\le d+1}\,\dd x_1\cdots\dd x_{d}.
\end{equation}
Then,
\[
 I(x_{d+1}) = (-1)^d\times \begin{cases}
 \pf\begin{pmatrix}
  A & p(x_{d+1}) \\ -p(x_{d+1})^\top & 0
\end{pmatrix}
 & (d:\,\mathrm{even}), \\
  \pf\begin{pmatrix}
  A & b & p(x_{d+1}) \\ -b^\top & 0 & 1 \\ -p(x_{d+1})^\top & -1 & 0
 \end{pmatrix}
 -\pf(A)
 & (d:\,\mathrm{odd}),
 \end{cases}
\]
where $A=(a_{ij})_{1\le i,j\le d+1}$, $b=(b_1,\ldots,b_{d+1})^\top$, $p(x_{d+1})=(\phi_1(x_{d+1}),\ldots,\phi_{d+1}(x_{d+1}))^\top$ with
\begin{equation}
\label{Ab}
\begin{aligned}
& a_{ij} = \int_{-\infty<x<y<\infty}\det\begin{pmatrix}
 \phi_i(x) & \phi_i(y) \\
 \phi_j(x) & \phi_j(y) \end{pmatrix}\,\dd x\,\dd y, \\
& b_i = \int_{-\infty}^\infty \phi_i(x)\,\dd x. 
\end{aligned}
\end{equation}
\end{theorem}

The original de Bruijn's integral is over $x_1<\cdots<x_d<x_{d+1}$.
We call (\ref{I}) the incomplete de Bruijn integral, since the variable $x_{d+1}$ is not used as an integration variable.

In the rest of this subsection, we prove Theorem \ref{thm:incomplete}.

\begin{proof}
Let $V$ be a vector space with basis $e_1,e_2,\ldots$,
and define a $V$-valued function
\[
  \phi(x) = \sum_{i=1}^{d+1} \phi_i(x) e_i.
\]
Then, the determinant is expressed as
\[
 \phi(x_1) \wedge\cdots\wedge \phi(x_{d+1}) =
 \det\bigl(\phi_i(x_j)\bigr)_{1\le i,j\le d+1}\,e_1\wedge\cdots\wedge e_{d+1}.
\]
Let
\[
 \Omega = \int_{-\infty<x<y<\infty}\phi(x)\wedge\phi(y)\,\dd x\,\dd y, \qquad
 \Omega_1 = \int_{-\infty}^{\infty}\phi(x)\,\dd x.
\]
The evaluation of $I(x_{d+1})$ depends on the parity of $d$.

Case 1: $d$ is even.
The integral $I(x_{d+1})$ is expressed as the coefficient of the ($d+2$)-form
\begin{equation}
\label{F_wedge}
 I(x_{d+1}) \, e_1\wedge\cdots\wedge e_{d+2} = \left(\int_{x_1<\cdots<x_{d}} \phi(x_1)\wedge\cdots\wedge\phi(x_{d})\,\dd x_1\cdots\,\dd x_{d}\right) \,\wedge\phi(x_{d+1})\wedge e_{d+2}.
\end{equation}
We rewrite the integral in (\ref{F_wedge}) as
\begin{equation}
\label{F_wedge0}
 \pm \int_{x_2<x_4<\cdots<x_{d}}
 \phi(x_2)\wedge\phi(x_4)\wedge\cdots\wedge\phi(x_{d})\wedge
 \Phi(x_2,x_4,\ldots,x_{d})\,\dd x_2\dd x_4\cdots\dd x_{d},
\end{equation}
where
\begin{equation}
\label{F_wedge1}
\begin{aligned}
& \Phi(x_2,x_4,\ldots,x_{d}) = \\
& \left( \int_{-\infty}^{x_2} \phi(x_1)\,\dd x_1 \right)
 \wedge\left( \int_{x_{2}}^{x_{4}} \phi(x_{3})\,\dd x_{3} \right)
 \wedge\left( \int_{x_{4}}^{x_{6}} \phi(x_{5})\,\dd x_{5} \right)
 \wedge\cdots
 \wedge\left( \int_{x_{d-2}}^{x_d} \phi(x_{d-1})\,\dd x_{d-1} \right),
\end{aligned}
\end{equation}

By adding the value
\begin{align*}
 &\left( \int_{-\infty}^{x_2}\phi(x_1)\,\dd x_1 \right)
 \wedge\left( \int_{-\infty}^{x_{2}}\phi(x_{3})\,\dd x_{3} \right)
 \wedge\left( \int_{x_{4}}^{x_{6}}\phi(x_{5})\,\dd x_{5} \right)
 \wedge\cdots
 \wedge\left( \int_{x_{d-2}}^{x_{d}}\phi(x_{d-1})\,\dd x_{d-1} \right) \\
 &= 0
\end{align*}
to (\ref{F_wedge1}), we obtain
\[
\begin{aligned}
& \Phi(x_2,x_4,\ldots,x_{d}) \\
& =\left( \int_{-\infty}^{x_2}\phi(x_1)\,\dd x_1 \right)
 \wedge\left( \int_{-\infty}^{x_4}\phi(x_{3})\,\dd x_{3} \right)
 \wedge\left( \int_{x_{4}}^{x_6}\phi(x_{5})\,\dd x_{5} \right)
 \wedge\cdots
 \wedge\left( \int_{x_{d-2}}^{x_{d}}\phi(x_{d-1})\,\dd x_{d-1} \right).
\end{aligned}
\]
By iterating this procedure, we finally obtain
\[
 \Phi(x_2,x_4,\ldots,x_{d})
=\left( \int_{-\infty}^{x_2}\phi(x_1)\,\dd x_1 \right)
 \wedge\left( \int_{-\infty}^{x_4}\phi(x_3)\,\dd x_3 \right)
 \wedge\cdots
 \wedge\left( \int_{-\infty}^{x_d}\phi(x_{d-1})\,\dd x_{d-1} \right),
\]
which is skew-symmetric in $(x_2,x_4,\ldots,x_{d})$.
Therefore, the integrand in (\ref{F_wedge0}) is symmetric in $(x_2,x_4,\ldots,x_{d})$, and hence we have
\begin{align*}
  (\mbox{\ref{F_wedge0}})
 =& \frac{\pm 1}{(d/2)!} \int_{-\infty}^\infty\cdots\int_{-\infty}^\infty
 \phi(x_2)\wedge\phi(x_4)\wedge\cdots\wedge\phi(x_{d})\wedge
 \Phi(x_2,x_4,\ldots,x_{d})\,\dd x_2\dd x_4\cdots\dd x_{d} \\
 =& \frac{1}{(d/2)!}
\left( \int_{x_1<x_2} \phi(x_1)\wedge\phi(x_2)\,\dd x_1\,\dd x_2\right) \wedge\cdots\wedge
\left(\int_{x_{d-1}<x_d} \phi(x_{d-1})\wedge\phi(x_d)\,\dd x_{d-1}\,\dd x_d\right) \\
 =& \frac{1}{(d/2)!} \Omega^{d/2}, \qquad
\Omega^{d/2}=\underbrace{\Omega\wedge\cdots\wedge\Omega}_{d/2\ \mathrm{times}}.
\end{align*}
Therefore,
\begin{align*}
 (\mbox{\ref{F_wedge}})
 = \frac{1}{(d/2)!} \Omega^{d/2} \wedge \phi(x_{d+1}) \wedge e_{d+2}
 =& \frac{1}{((d+2)/2)!} (\Omega + \phi(x_{d+1}) \wedge e_{d+2})^{(d+2)/2} \\
 =&
 \pf\begin{pmatrix}
  A & p(x_{d+1}) \\ -p(x_{d+1})^\top & 0
\end{pmatrix}
 e_1\wedge\cdots\wedge e_{d+2},
\end{align*}
where $A$ and $p(x_{d+1})$ are defined in (\ref{Ab}).

Case 2: $d$ is odd.
The integral $I(x_{d+1})$ is expressed as the coefficient of the ($d+3$)-form
\begin{equation}
\label{F_wedge3}
 I(x_{d+1}) \, e_1\wedge\cdots\wedge e_{d+3} = \left(\int_{x_1<\cdots<x_{d}} \phi(x_1)\wedge\cdots\wedge\phi(x_{d})\,\dd x_1\cdots \dd x_{d}\right) \,\wedge\phi(x_{d+1}) \wedge e_{d+2} \wedge e_{d+3}.
\end{equation}
We rewrite the integral in (\ref{F_wedge3}) as
\begin{equation}
\label{F_wedge4}
 \pm \int_{x_2<x_4<\cdots<x_{d-1}}
 \phi(x_2)\wedge\phi(x_4)\wedge\cdots\wedge\phi(x_{d-1})\wedge
 \Phi(x_2,x_4,\ldots,x_{d-1})\,\dd x_2\dd x_4\cdots\dd x_{d-1},
\end{equation}
where
\begin{equation*}
 \Phi(x_2,x_4,\ldots,x_{d-1})
=\left( \int_{-\infty}^{x_2} \phi(x_1)\,\dd x_1 \right)
 \wedge\left( \int_{x_{2}}^{x_{4}} \phi(x_{3})\,\dd x_{3} \right)
 \wedge\cdots
 \wedge\left( \int_{x_{d-1}}^{\infty} \phi(x_d)\,\dd x_d \right).
\end{equation*}
As in the case where $d$ is even, this integral is modified as
\[
 \Phi(x_2,x_4,\ldots,x_{d-1})
=\left( \int_{-\infty}^{x_2}\phi(x_1)\,\dd x_1 \right)
 \wedge\cdots
 \wedge\left( \int_{-\infty}^{x_{d-1}}\phi(x_{d-2})\,\dd x_{d-2} \right)
 \wedge\left( \int_{-\infty}^{\infty}\phi(x_d)\,\dd x_d \right),
\]
which is skew-symmetric in $(x_2,x_4,\ldots,x_{d-1})$.
Since the integrand in (\ref{F_wedge4}) is symmetric in $(x_2,x_4,\ldots,x_{d-1})$,
\begin{align*}
  (\mbox{\ref{F_wedge4}})
 =& \frac{\pm 1}{((d-1)/2)!} \\
  & \times\int_{-\infty}^\infty\cdots\int_{-\infty}^\infty
 \phi(x_2)\wedge\phi(x_4)\wedge\cdots\wedge\phi(x_{d-1})\wedge
 \Phi(x_2,x_4,\ldots,x_{d-1}) \,\dd x_2\dd x_4\cdots\dd x_{d-1} \\
 =& \frac{1}{((d-1)/2)!} \Omega^{(d-1)/2}\wedge\Omega_1.
\end{align*}
Therefore,
\begin{equation}
\label{F_wedge6}
 (\mbox{\ref{F_wedge3}})
 = \frac{-1}{((d-1)/2)!} \Omega^{(d-1)/2}\wedge \Omega_1\wedge e_{d+2} \wedge \phi(x_{d+1})\wedge e_{d+3}.
\end{equation}
Moreover, noting the expansion
\begin{align*}
& (\Omega + \Omega_1\wedge e_{d+2} + \phi(x_{d+1})\wedge e_{d+3} + e_{d+2}\wedge e_{d+3})^{(d+3)/2} \\
&= \frac{d+3}{2} \Omega^{(d+1)/2}\wedge (\Omega_1\wedge e_{d+2} + \phi(x_{d+1})\wedge e_{d+3} + e_{d+2}\wedge e_{d+3}) \\
&\quad + \frac{(d+3)(d+1)}{8} \Omega^{(d-1)/2}\wedge (\Omega_1\wedge e_{d+2} + \phi(x_{d+1})\wedge e_{d+3} + e_{d+2}\wedge e_{d+3})^2 \\
&= \frac{d+3}{2} \Omega^{(d+1)/2}\wedge e_{d+2}\wedge e_{d+3} \\
&\quad + \frac{(d+3)(d+1)}{4} \Omega^{(d-1)/2}\wedge \Omega_1\wedge e_{d+2}\wedge \phi(x_{d+1})\wedge e_{d+3},
\end{align*}
we have
\begin{equation*}
\begin{aligned}
  (\mbox{\ref{F_wedge6}})
=& 
 \frac{-1}{((d+3)/2)!} (\Omega + \Omega_1\wedge e_{d+2} + \phi(x_{d+1})\wedge e_{d+3} + e_{d+2}\wedge e_{d+3})^{(d+3)/2} \\
 &+ \frac{1}{((d+1)/2)!} \Omega^{(d+1)/2}\wedge e_{d+2}\wedge e_{d+3} \\
=&
-\left( \pf\begin{pmatrix}
  A & b & p(x_{d+1}) \\ -b^\top & 0 & 1 \\ -p(x_{d+1})^\top & -1 & 0
\end{pmatrix}
 -\pf(A) \right)
 e_1\wedge\cdots\wedge e_{d+3},
\end{aligned}
\end{equation*}
where $A$, $b$, and $p(x_{d+1})$ are defined in (\ref{Ab}).
This completes the proof.
\end{proof}

\subsection{$\bm{G_d(x;\pm 1;0)}$ when $d$ is even}
\label{subsec:pm1}

We assume $d$ is even.
Let $p_n(x)$ be a skew-orthogonal polynomial of degree $n$ such that
\begin{equation}
\begin{aligned}
 \int_{x<y} (p_i(x)p_j(y)-p_j(x)p_i(y)) e^{-\frac{1}{2}x^2-\frac{1}{2}y^2} \dd x\dd y
 =&  \int_{\R^2} \sgn(y-x) p_i(x)p_j(y) e^{-\frac{1}{2}x^2-\frac{1}{2}y^2} \dd x\dd y \\
 =&  \begin{cases}
 \sigma_k  & (i,j)=(2k,2k+1), \\
 -\sigma_k & (i,j)=(2k+1,2k), \\
 0         & (\mbox{otherwise}),
   \end{cases}
\label{skew_orthogonality}
\end{aligned}
\end{equation}
where $k=0,\ldots,d/2-1$ is an integer.
It is well known that
\begin{equation}
\label{pn}
 p_n(x) = \begin{cases}
\displaystyle
  2^{-n/2} H_n(\sqrt{2}x) & (n \mbox{ is even}), \\
\displaystyle
  2^{-n/2} H_n(\sqrt{2}x) - \frac{n-1}{2} 2^{-(n-2)/2} H_{n-2}(\sqrt{2}x) & (n \mbox{ is odd}),
          \end{cases}
\qquad \sigma_k = 2\sqrt{\pi}\frac{(2k)!}{2^{2k}}, 
\end{equation}
form such a system \cite{Nagao-Wadati:1991,Kuriki:arXiv}.

Using this monic system $p_n(x)$, we apply Theorem \ref{thm:G_goe}.
When $z=-1$, $v_{ij}$ in (\ref{vij}) becomes 
the left-hand side of (\ref{skew_orthogonality}).
Hence, we have
\begin{equation}
\label{V}
 V = (v_{ij}) = \diag\left(
 \begin{pmatrix} 0 & \sigma_0 \\ -\sigma_0 & 0 \end{pmatrix},\ldots,
 \begin{pmatrix} 0 & \sigma_{d/2-1} \\ -\sigma_{d/2-1} & 0 \end{pmatrix},0 \right)
\end{equation}
and
\begin{equation}
\label{Gd-1}
\begin{aligned}
 G_d(x;-1;0)
 =& \frac{1}{c_d}\pf\begin{pmatrix}
  V & p(x) \\ -p(x)^\top & 0 \end{pmatrix} \\
 =& \frac{1}{c_d} \pf\left(\begin{array}{ccccccc}
  0 & \multicolumn{1}{c|}{\sigma_0} & & & & & p_0(x) \\
  -\sigma_0 & \multicolumn{1}{c|}{0} & & & & & p_1(x) \\ \cline{1-2}
  & & \ddots & & & & \vdots \\ \cline{4-5}
  & & & \multicolumn{1}{|c}{0} & \multicolumn{1}{c|}{\sigma_{d/2-1}} & & p_{d-2}(x) \\
  & & & \multicolumn{1}{|c}{-\sigma_{d/2-1}} & \multicolumn{1}{c|}{0} & & p_{d-1}(x) \\ \cline{4-7}
  & & & & & \multicolumn{1}{|c}{0} & p_d(x) \\
  -p_0(x) & -p_1(x) & \cdots & -p_{d-2}(x) & -p_{d-1}(x) & \multicolumn{1}{|c}{-p_d(x)} & 0
 \end{array}\right) \\
 =& \frac{1}{c_d} \left(\prod_{k=0}^{d/2-1}\sigma_k\right) p_d(x) =  p_d(x) =  2^{-d/2} H_d(\sqrt{2}x),
\end{aligned}
\end{equation}
which is equal to (\ref{z=-1}) when $d$ is even.

Let
\[
 q_n(y)=\int_y^\infty p_n(t) e^{-t^2/2} \dd t,
\qquad
 r_n=\int_{-\infty}^\infty p_n(t) e^{-t^2/2} \dd t =
 \begin{cases}
\displaystyle
 2^{-n}\sqrt{2\pi}\frac{n!}{(n/2)!} & (n\mbox{: even}), \\
 0 & (n\mbox{: odd}),
 \end{cases}
\]
and
\[
 q(y)=(q_0(y),\ldots,q_{d}(y))^\top, \quad
 r=(r_0,0,r_2,0,\ldots,r_{d})^\top.
\]
Then, $w^+_i(y)=q_{i-1}(y)$, $w^-_i(y)=r_{i-1}-q_{i-1}(y)$. 
Define
\[
 L_d(x,y) = \frac{1}{c_d}\pf\begin{pmatrix}
  V - 2 (r q(y)^\top-q(y) r^\top) & p(x) \\ -p(x)^\top & 0 \end{pmatrix},
\]
where $V$ was defined in (\ref{V}).
Then, Theorem \ref{thm:G_goe} states that
\[
 G_d(c;1;0)=L_d(x,y)|_{x=y=c}.
\]
For $x$ such that $p_d(x)\ne 0$, by means of (\ref{2rank}),
\begin{equation}
\label{Ld}
\begin{aligned}
L_d(x,y)
 &= \frac{1}{c_d}\pf\begin{pmatrix}
  V & p(x) \\ -p(x)^\top & 0 \end{pmatrix}\left[1 - 2(q(y)^\top,0)
\begin{pmatrix}
  V & p(x) \\ -p(x)^\top & 0 \end{pmatrix}^{-1} \begin{pmatrix}r \\ 0 \end{pmatrix}\right] \\
 &= 2^{-d/2}H_d(\sqrt{2}x) + K_d(x,y)
\end{aligned}
\end{equation}
where
\begin{equation}
\label{Kd}
\begin{aligned}
  K_d(x,y)
 =& L_d(x,y)-2^{-d/2}H_d(\sqrt{2}x) \\
 =& 2\,p_d(x) 
 (-q(y)^\top,0) \begin{pmatrix} V & p(x) \\ -p(x)^\top & 0 \end{pmatrix}^{-1} \begin{pmatrix}r \\ 0 \end{pmatrix} \quad (\mbox{when }p_d(x)\ne 0).
\end{aligned}
\end{equation}
(When $p_d(x)\ne 0$, the inverse matrix exists by (\ref{Gd-1})).

Now it is sufficient to show that
\begin{equation}
\label{KdKd}
 K_d(x,y) = \widetilde K_d(x,y),
\end{equation}
where
\begin{align*}
 \widetilde K_d(x,y) = 
 \sqrt{\frac{1}{2\pi}}\,\left(\frac{d-1}{2}\right)!\Biggl[ &
 2 \sum_{k=0}^{d}\frac{H_{k}(\sqrt{2}x)e^{-y^2/2}H_{k}(\sqrt{2}y)}{k!} \\
& +\frac{1}{2^{1/2}d!}\,H_d(\sqrt{2}x)\,
\int_{-\infty}^{\infty} e^{-t^2/2}\,H_{d+1}(\sqrt{2}t)\,\sgn(y-t)\,\dd t
\Biggr].
\end{align*}
If (\ref{KdKd}) holds for $x$ such that $p_d(x)\ne 0$, then it holds for all $x$ by continuity.
We prove (\ref{KdKd}) by checking
\[
 K_d(x,\infty) = \widetilde K_d(x,\infty) = 0
\]
and
\[
 \frac{\partial K_d(x,y)}{\partial y} = \frac{\partial \widetilde K_d(x,y)}{\partial y}.
\]

Using the explicit form of the inverse matrix in (\ref{Ld}) or (\ref{Kd}):
\begin{equation*}
 \begin{pmatrix} V & p(x) \\ -p(x)^\top & 0 \end{pmatrix}^{-1} =
 \frac{-1}{p_d(x)} \left(\begin{array}{ccccccc}
 0 & \multicolumn{1}{c|}{\frac{p_d(x)}{\sigma_0}} & & & & -\frac{p_1(x)}{\sigma_0} \\    
 -\frac{p_d(x)}{\sigma_0} & \multicolumn{1}{c|}{0} & & & & \frac{p_0(x)}{\sigma_0} \\ \cline{1-2}    
 & & \ddots & & & \vdots \\ \cline{4-5}
 & & & \multicolumn{1}{|c}{0} & \multicolumn{1}{c|}{\frac{p_d(x)}{\sigma_{d/2-1}}}  & -\frac{p_{d-1}(x)}{\sigma_{d/2-1}} \\
 & & & \multicolumn{1}{|c}{-\frac{p_d(x)}{\sigma_{d/2-1}}} & \multicolumn{1}{c|}{0} & \frac{p_{d-2}(x)}{\sigma_{d/2-1}} \\ \cline{4-7}
\frac{p_1(x)}{\sigma_0} & -\frac{p_0(x)}{\sigma_0} & \cdots & \frac{p_{d-1}(x)}{\sigma_{d/2-1}}  & -\frac{p_{d-2}(x)}{\sigma_{d/2-1}} & \multicolumn{1}{|c}{0} & 1 \\
 & & & & & \multicolumn{1}{|c}{-1} & 0
 \end{array}\right),
\end{equation*}
together with the definition of $p_n(x)$ in (\ref{pn}), the identity $r_{2k}/\sigma_k=1/(\sqrt{2}k!)$ and the recurence relation for the Hermite polynomial, we can verify that
\[
  \frac{\partial K_d(x,y)}{\partial y} = 2\,p_d(x) 
 e^{-y^2/2} (p(y)^\top,0) \begin{pmatrix} V & p(x) \\ -p(x)^\top & 0 \end{pmatrix}^{-1} \begin{pmatrix}r \\ 0 \end{pmatrix}
\]
and
\begin{align*}
 \frac{\partial \widetilde K_d(x,y)}{\partial y} = 
 \sqrt{\frac{1}{2\pi}}\,\left(\frac{d-1}{2}\right)!\Biggl[ &
 2 \sum_{k=0}^{d}\frac{H_{k}(\sqrt{2}x)e^{-y^2/2} \left[-2^{-1/2} H_{k+1}(\sqrt{2}y) + 2^{-1/2} k H_{k-1}(\sqrt{2}y)\right]}{k!} \\
& +\frac{2}{2^{1/2}d!}\,H_d(\sqrt{2}x)\, e^{-y^2/2}\,H_{d+1}(\sqrt{2}y)
\Biggr]
\end{align*}
have the same expansion $e^{-y^2/2}\sum_{i,j\ge 0} t_{ij} H_{i}(\sqrt{2}x) H_{j}(\sqrt{2}y)$ with the same coefficients $t_{ij}$.

Now we have proven that
\begin{equation}
\label{LdHKd}
 L_d(x,y) = 2^{-d/2}H_d(\sqrt{2}x) + \widetilde K_d(x,y)
\end{equation}
and hence
\[
G_d(x;1;0) = L_d(x,x) = 2^{-d/2}H_d(\sqrt{2}x) + \widetilde K_d(x,x),
\]
which is equal to (\ref{z=1}) when $d$ is even.

\subsection{$\bm{F_d(\nu;z)}$ for $\bm{d=1,2}$}
\label{subsec:d=1,2}

We show that when $d=1,2$, the integral in (\ref{main}) can be evaluated analytically and $F_d(\nu;z)$ is expressed without integration.

Let
\begin{align*}
 J^\pm_k(a,b) = \int_{-\infty}^\infty e^{-\frac{1}{2}\xi^2} e^{-a w^2} w^k \erfc\bigl(\pm\sqrt{b}\,w\bigr) \,\dd\xi,
\end{align*}
where
\begin{equation}
\label{w}
 w = \frac{1}{\sqrt{2}}\Bigl(\sqrt{2\sigma}\xi+\nu\sqrt{1-2\sigma}\Bigr).
\end{equation}
This includes
\[
 J^\pm_k(a,0) =  J_k(a,0) = \int_{-\infty}^\infty e^{-\frac{1}{2}\xi^2} e^{-a w^2} w^k \,\dd\xi
\]
as a special case.
Then,
\begin{equation}
\label{F1}
\begin{aligned}
F_1(\nu;z)
 =& \Bigl(\frac{\gamma}{2\pi}\Bigr)^{1/2} \phi(\nu)
 \Bigl(\frac{1-2\sigma}{2}\Bigr)^{-1/2} \int_{-\infty}^\infty \phi(\xi) G_1(w;z;0) \,\dd\xi \\
 =& \Bigl(\frac{\gamma}{2\pi}\Bigr)^{1/2} 
 \phi(\nu)\Bigl(\frac{1-2\sigma}{2}\Bigr)^{-1/2} \frac{1}{\sqrt{2\pi}} \\
&\times \biggl[
 \biggl(\frac{J_0(1/2,0)}{\sqrt{2\pi}}
-\frac{J^+_1(0,1/2)}{2}
\biggr)
 +\biggl(
 \frac{J_0(1/2,0)}{\sqrt{2\pi}}
 +\frac{J^-_1(0,1/2)}{2}
\biggr)z \biggr],
\end{aligned}
\end{equation}
\begin{equation}
\label{F2}
\begin{aligned}
F_2(\nu;z)
 =& \Bigl(\frac{\gamma}{2\pi}\Bigr)^{2/2} \phi(\nu)
 \Bigl(\frac{1-2\sigma}{2}\Bigr)^{-1} \int_{-\infty}^\infty \phi(\xi) G_2(w;z;0) \,\dd\xi \\
 =& \Bigl(\frac{\gamma}{2\pi}\Bigr) 
 \phi(\nu)\Bigl(\frac{1-2\sigma}{2}\Bigr)^{-1} \frac{1}{\sqrt{2\pi}} \\
&\times \biggl[
 \biggl(-\frac{J_1(1,0)}{2\sqrt{\pi}}
+\frac{J^+_0(1/2,1/2)}{2\sqrt{2}}
+\frac{2 J^+_2(0,1)-J^+_0(0,1)}{4}
\biggr)
 + \frac{J_0(1/2,0)}{\sqrt{2}} z \\
& \quad +\biggl(
\frac{J_1(1,0)}{2\sqrt{\pi}}
+\frac{J^-_0(1/2,1/2)}{2\sqrt{2}}
+\frac{2 J^-_2(0,1)-J^-_0(0,1)}{4}
\biggr)z^2 \biggr].
\end{aligned}
\end{equation}

We obtain $J^\pm_k$ using a recurrence formula in $k$.
We first assume that $\sigma\ge 0$ and hence $w$ in (\ref{w}) is real.
Recalling (\ref{w}), we use the change of measure
\[
 e^{-\frac{1}{2}\xi^2} e^{-a w^2}\,\dd\xi
= e^{-\frac{a (1-2\sigma)\nu^2}{2(1+2\sigma a)}} e^{-\frac{1}{2}\eta^2}
 \frac{\dd\eta}{\sqrt{1+2\sigma a}}
\]
where
\[
 \eta=\sqrt{1+2\sigma a}\biggl(\xi+\frac{a \sqrt{2\sigma}\sqrt{1-2\sigma}}{1+2\sigma a}\nu\biggr).
\]

We first obtain the recurrence formula for $J^\pm_k$ by integration by parts.
Noting that
\begin{align*}
\sqrt{2}w
= \sqrt{2\sigma} \biggl(\frac{\eta}{\sqrt{1+2\sigma a}}-\frac{a \sqrt{2\sigma}\sqrt{1-2\sigma}}{1+2\sigma a}\nu \biggr) + \sqrt{1-2\sigma}\nu
= \frac{\sqrt{2\sigma}}{\sqrt{1+2\sigma a}} \eta +\frac{\sqrt{1-2\sigma}}{1+2\sigma a}\nu,
\end{align*}
we have
\begin{align*}
\sqrt{2}\int_{-\infty}^\infty e^{-\frac{1}{2}\xi^2} e^{-a w^2} w^k F(w) \,\dd\xi
=&
  \frac{\sqrt{2\sigma}}{\sqrt{1+2\sigma a}} \int_{-\infty}^\infty e^{-\frac{1}{2}\xi^2} e^{-a w^2} \eta w^{k-1} F(w) \,\dd\xi \\
& +\frac{\sqrt{1-2\sigma}}{1+2\sigma a}\nu \int_{-\infty}^\infty e^{-\frac{1}{2}\xi^2} e^{-a w^2} w^{k-1} F(w) \,\dd\xi,
\end{align*}
whose first term is
\begin{align*}
& \frac{\sqrt{2\sigma}}{\sqrt{1+2\sigma a}} \frac{1}{\sqrt{1+2\sigma a}}
e^{-\frac{a (1-2\sigma)\nu^2}{2(1+2\sigma a)}}
 \int_{-\infty}^\infty e^{-\frac{1}{2}\eta^2} \eta w^{k-1} F(w)\,\dd\eta \\
&\qquad = \frac{\sqrt{2\sigma}}{\sqrt{1+2\sigma a}} \frac{1}{\sqrt{1+2\sigma a}}
e^{-\frac{a (1-2\sigma)\nu^2}{2(1+2\sigma a)}}
 \int_{-\infty}^\infty e^{-\frac{1}{2}\eta^2} \frac{d w}{d \eta} \frac{d}{d w}( w^{k-1} F(w))\,\dd\eta \\
&\qquad = \frac{2\sigma}{\sqrt{2}(1+2\sigma a)}\int_{-\infty}^\infty e^{-\frac{1}{2}\xi^2} e^{-a w^2} \frac{d}{d w} (w^{k-1} F(w)) \,\dd\xi.
\end{align*}
Therefore,
\begin{equation*}
\begin{aligned}
\int_{-\infty}^\infty e^{-\frac{1}{2}\xi^2} e^{-a w^2} w^k F(w) \,\dd\xi
=&
  \frac{2\sigma}{2(1+2\sigma a)}\int_{-\infty}^\infty e^{-\frac{1}{2}\xi^2} e^{-a w^2} w^{k-1} F'(w) \,\dd\xi \\
&+ (k-1)\frac{2\sigma}{2(1+2\sigma a)}\int_{-\infty}^\infty e^{-\frac{1}{2}\xi^2} e^{-a w^2} w^{k-2} F(w) \,\dd\xi \\
& +\frac{\sqrt{1-2\sigma}}{\sqrt{2}(1+2\sigma a)}\nu \int_{-\infty}^\infty e^{-\frac{1}{2}\xi^2} e^{-a w^2} w^{k-1} F(w) \,\dd\xi.
\end{aligned}
\end{equation*}
For $F(w)=\erfc\bigl(\pm\sqrt{b}w\bigr)$,
\[
 F'(w)=\mp\frac{2\sqrt{b}}{\sqrt{\pi}} e^{-b w^2}
\]
and
\begin{equation}
\label{rule22}
\begin{aligned}
J^\pm_k(a,b)
=&
\int_{-\infty}^\infty e^{-\frac{1}{2}\xi^2} e^{-a w^2} w^k \erfc\bigl(\pm\sqrt{b}\,w\bigr) \,\dd\xi \\
=&
 \mp\frac{2\sigma}{1+2\sigma a}\frac{\sqrt{b}}{\sqrt{\pi}} J_{k-1}(a+b,0) \\
&+ (k-1)\frac{2\sigma}{2(1+2\sigma a)} J^\pm_{k-2}(a,b)
 + \frac{\sqrt{1-2\sigma}}{\sqrt{2}(1+2\sigma a)}\nu J^\pm_{k-1}(a,b).
\end{aligned}
\end{equation}

Next, we obtain the initial value $J^\pm_0(a,b)$.
Recalling that $\erfc(x) = 2\Pr\bigl(\xi>\sqrt{2}x\bigr)$ with $\xi\sim \mathcal{N}(0,1)$,
we have
\begin{align*}
\int_{-\infty}^\infty & e^{-\frac{1}{2}\xi^2} e^{-a w^2} \erfc\bigl(\pm\sqrt{b}\,w\bigr) \,\dd\xi \\
&=\frac{1}{\sqrt{1+2\sigma a}}
e^{-\frac{a (1-2\sigma)\nu^2}{2(1+2\sigma a)}}\int_{-\infty}^\infty e^{-\frac{1}{2}\eta^2}\dd\eta \times 2\Pr\bigl(\xi>\pm\sqrt{2}(\sqrt{b}\,w)\bigr) \\
&=\frac{\sqrt{2\pi}}{\sqrt{1+2\sigma a}}
e^{-\frac{a (1-2\sigma)\nu^2}{2(1+2\sigma a)}} \times 2\Pr\bigl(\xi>\pm\sqrt{2 b}\,W\bigr)
\end{align*}
with
\[
 \sqrt{2}W = \frac{\sqrt{2\sigma}}{\sqrt{1+2\sigma a}} \eta + \frac{\sqrt{1-2\sigma}}{1+2\sigma a}\nu,
\]
where $\eta\sim \mathcal{N}(0,1)$ is independent of $\xi$.
Therefore,
\begin{align*}
 2\Pr\bigl(\xi>\sqrt{2 b}\,W\bigr)
=& 2\Pr\biggl(\xi \mp\frac{\sqrt{b} \sqrt{2\sigma}}{\sqrt{1+2\sigma a}}\eta > \pm\frac{\sqrt{b}\sqrt{1-2\sigma}}{1+2\sigma a}\nu\biggr) \\
=& 2\Pr\biggl(\xi' > \pm\frac{\sqrt{b}\sqrt{1-2\sigma}}{\sqrt{1+2\sigma a}\sqrt{1+(a+b) 2\sigma}}\nu\biggr), \quad \xi'\sim\mathcal{N}(0,1), \\
=& \erfc\biggl(\pm\frac{\sqrt{b/2}\sqrt{1-2\sigma}}{\sqrt{1+2\sigma a}\sqrt{1+(a+b) 2\sigma}}\nu\biggr), 
\end{align*}
hence,
\begin{equation}
\label{rule3}
\begin{aligned}
 J^\pm_0(a,b)
=& \int_{-\infty}^\infty e^{-\frac{1}{2}\xi^2} e^{-a w^2} \erfc\bigl(\pm\sqrt{b}\,w\bigr) \,\dd\xi \\
=& \frac{\sqrt{2\pi}}{\sqrt{1+2\sigma a}}
 e^{-\frac{a (1-2\sigma)\nu^2}{2(1+2\sigma a)}}
\erfc\biggl(\pm\frac{\sqrt{b/2}\sqrt{1-2\sigma}}{\sqrt{1+2\sigma a}\sqrt{1+(a+b) 2\sigma}}\nu\biggr). 
\end{aligned}
\end{equation}

Although we assumed $\sigma\ge 0$ so far,
(\ref{rule22}) and (\ref{rule3}) hold for admissible negative values of $\sigma$ by analytic continuation.

By means of (\ref{rule22}) and (\ref{rule3}), the integral
$\int_{-\infty}^\infty e^{-\frac{1}{2}\xi^2} G_d(w;z;0) \,\dd\xi$ in (\ref{F1}) and (\ref{F2}) can be evaluated.
Now we obtain
\begin{align*}
 F_1(\nu;z) = \Bigl(\frac{\gamma}{2\pi}\Bigr)^{1/2} \phi(\nu) \times\Biggl[ &
\Biggl(\frac{\sqrt{2+2\sigma} e^{-\frac{(1-2\sigma) \nu^2}{2 (2+2\sigma)}}}{\sqrt{2 \pi } \sqrt{1-2\sigma}}-\frac{1}{2} \nu \erfc\left(\frac{\sqrt{1-2\sigma} \nu}{\sqrt{2} \sqrt{2+2\sigma}}\right)\Biggr) \\
& + z \Biggl(
\frac{\sqrt{2+2\sigma} e^{-\frac{(1-2\sigma) \nu^2}{2 (2+2\sigma)}}}{\sqrt{2 \pi } \sqrt{1-2\sigma}}
+ \frac{1}{2} \nu \erfc\left(-\frac{\sqrt{1-2\sigma} \nu}{\sqrt{2} \sqrt{2+2\sigma}}\right)\Biggr)
\Biggr],
\end{align*}
\begin{align*}
 F_2(\nu;z) = \Bigl(\frac{\gamma}{2\pi}\Bigr) \phi(\nu)
 \times\Biggl[ &
 \Biggl(
 -\frac{\sqrt{1+2\sigma}\nu e^{-\frac{(1-2\sigma) \nu^2}{2 (1+2\sigma)}}}{\sqrt{2\pi}\sqrt{1-2\sigma}}
 +\frac{e^{-\frac{(1-2\sigma) \nu^2}{2 (2+2\sigma)}} \erfc\left(\frac{\sqrt{1-2\sigma} \nu}{\sqrt{2} \sqrt{1+2\sigma} \sqrt{2+2\sigma}}\right)}{(1-2\sigma) \sqrt{2+2\sigma}} \\
 & \qquad\qquad + \frac{1}{2} \left(\nu^2-1\right) \erfc\left(\frac{\sqrt{1-2\sigma} \nu}{\sqrt{2} \sqrt{1+2\sigma}}\right)
\Biggr) \\
 & +z\frac{2 e^{-\frac{(1-2\sigma) \nu^2}{2 (2+2\sigma)}}}{(1-2\sigma) \sqrt{2+2\sigma}} \\
& +z^2 \Biggl(
 \frac{\sqrt{1+2\sigma}\nu e^{-\frac{(1-2\sigma) \nu^2}{2 (1+2\sigma)}}}{\sqrt{2\pi}\sqrt{1-2\sigma}}
 +\frac{e^{-\frac{(1-2\sigma) \nu^2}{2 (2+2\sigma)}} \erfc\left(-\frac{\sqrt{1-2\sigma}\nu}{\sqrt{2} \sqrt{1+2\sigma} \sqrt{2+2\sigma}}\right)}{(1-2\sigma) \sqrt{2+2\sigma}} \\
 & \qquad\qquad + \frac{1}{2} \left(\nu^2-1\right) \erfc\left(-\frac{\sqrt{1-2\sigma} \nu}{\sqrt{2} \sqrt{1+2\sigma}}\right)
 \Biggr)
 \Biggr].
\end{align*}

The formulas for $F_d(\nu;z)$ above can be shown to be equivalent to those of \citet[Example 4.6]{Cheng-Schwartzman:2018} with $\kappa = \sqrt{1-2\sigma}$ and $\eta = \sqrt{2/\gamma}\sqrt{1-2\sigma}$.
We also confirm that
\begin{align*}
 F_1(\nu;-1)
 =& -\Bigl(\frac{\gamma}{2\pi}\Bigr)^{1/2} \phi(\nu)\nu, \\
 F_2(\nu;-1)
 =& \Bigl(\frac{\gamma}{2\pi}\Bigr) \phi(\nu)(\nu^2-1).
\end{align*}

The results for $F_d(\nu;z)$ are valid only for $\sigma\in(-1/d,1/2)$.
Howevere, they have the limits when $\sigma\downarrow -1/d$\,:
\begin{equation}
\label{F12boundary}
\begin{aligned}
 \lim_{\sigma\downarrow -1}F_1(\nu;z)
 =& \Bigl(\frac{\gamma}{2\pi}\Bigr)^{1/2} \phi(\nu)\bigl[-\nu \1(\nu<0) + z \nu 1(\nu>0)\bigr], \\
\lim_{\sigma\downarrow -1/2}F_2(\nu;z)
 =& \Bigl(\frac{\gamma}{2\pi}\Bigr) \phi(\nu)\Bigl[(e^{-\nu^2}+\nu^2-1)\1(\nu<0) + z e^{-\nu^2} + z^2(e^{-\nu^2}+\nu^2-1)\1(\nu>0)\Bigr],
\end{aligned}
\end{equation}
which are consistent with the formulas $G_d(x;z;-1/d)$ in (\ref{G12boundary}).

\subsection*{Acknowledgments}

This research was partially supported by JSPS KAKENHI Grants Nos.\ JP25K15034 (SK) and JP24K00634 (SI).

\bigskip

\bibliographystyle{abbrvnat}
\bibliography{index-bib}

\end{document}